\documentclass{article}
\usepackage{graphicx} % Required for inserting images
\graphicspath{{Figures/}}

\usepackage{hyperref}
\usepackage{amssymb}
\usepackage{mathtools}          
\usepackage{stmaryrd}
\usepackage{xstring}           
\usepackage{enumitem}          
\usepackage[f]{esvect} 
\usepackage{xcolor}            
\usepackage{yhmath}  
\usepackage{amsmath,amsthm}

\definecolor{sina}{rgb}{1.0,.0,1.0}         % What Sina has changed

\definecolor{chris}{rgb}{0, 0., 1}

\newtheorem{lemma}{Lemma}[section]
\newtheorem{proposition}{Proposition}[section]
\newtheorem{corollary}{Corollary}[section]

\newtheorem{remark}{Remark}

\newcommand{\diff}{\,{\rm d}}

\newcommand{\Cay}{\,{\rm Cay}}

\newcommand{\norm}[1]{\lVert#1\rVert}
\newcommand{\R}{\mathbb{R}}
\newcommand{\C}{\mathbb{C}}
\newcommand{\Om}{\Omega}

\title{Commutator-free Cayley methods}
\author{B.~Wembe,  C.~Offen, S.~Maslovskaya, S.~Ober-Blöbaum, P.~Singh}
\date{}

\begin{document}

\maketitle


% \author[label1]{B.~Wembe}
% \author[label2]{C.~Offen}
% \author[label3]{S.~Maslovskaya}
% \author[label4]{S.~Ober-Blöbaum}
% \author[label5]{P.~Singh}

% \affiliation[label1]{organization={Paderborn University},
% addressline={Department of Mathematics, Warburger Str. 100},
% city={Paderborn},
% postcode={33098},
% %state={}
% country={Germany, wboris@math.upb.de}}

% \affiliation[label2]{organization={University of Birmingham},
% addressline={School of Mathematics},
% city={Birmingham},
% postcode={B15 2TT},
% %state={}
% country={United Kingdom, c.offen@bham.ac.uk}}

% \affiliation[label3]{organization={Paderborn University},
% addressline={Department of Mathematics, Warburger Str. 100},
% city={Paderborn},
% postcode={33098},
% %state={}
% country={Germany, sofyam@math.upb.de}}

% \affiliation[label4]{organization={Paderborn University},
% addressline={Department of Mathematics, Warburger Str. 100},
% city={Paderborn},
% postcode={33098},
% %state={}
% country={Germany, sinaober@math.upb.de}}

% \affiliation[label5]{organization={Department of Mathematical Sciences, University of Bath},
% addressline={Department of Mathematics, Warburger Str. 100},
% city={Bath},
% postcode={BA2 7AY},
% %state={}
% country={United Kingdom, ps2106@bath.ac.uk}}

%% Abstract
\begin{abstract}
Differential equations posed on quadratic matrix Lie groups arise in the context of classical mechanics and quantum dynamical systems. Lie group numerical integrators preserve the constants of motions defining the Lie group. Thus, they respect important physical laws of the dynamical system, such as unitarity and energy conservation in the context of quantum dynamical systems, for instance.
In this article we develop a high-order commutator free Lie group integrator for non-autonomous differential equations evolving on quadratic Lie groups. Instead of matrix exponentials, which are expensive to evaluate and need to be approximated by appropriate rational functions in order to preserve the Lie group structure, the proposed method is obtained as a composition of Cayley transforms which naturally respect the structure of quadratic Lie groups while being computationally efficient to evaluate. Unlike Cayley--Magnus methods the method is also free from nested matrix commutators.
\end{abstract}

%% Keywords

\maketitle

%--------------------------------------------------
\section{Introduction}
In this article we are concerned with non-autonomous linear ordinary and partial differential equations,
\begin{equation}
\label{eq:pb}
\dot{Y}(t) = A(t)Y(t), \qquad Y(t_0) = Y_0,  \qquad t \in [t_0, t_f], 
\end{equation}
where the solution $Y(t) \in G$ evolves on Lie groups of the form
\begin{equation}
\label{eq:Lie group}
G = \{X \in \mathrm{GL}_n(\C) \, : \, XJX^* = J\},
\end{equation}
where $\mathrm{GL}_n(\C)$ is the group of $n\times n$ non-singular complex matrices and where $J \in \mathrm{GL}_n(\C)$ is a given matrix (see for instance \cite{BC:2017, CCLMO2022, HLW:2006}). 
The condition $Y(t) \in G$ for all $t$ is fulfilled if and only if $\dot Y(t)=A(t)Y(t)$ is tangential to the Lie group $G$. This is the case if and only if $A(t)$ takes values in the Lie algebra $\mathfrak g = \{ \Omega \in \C^{n \times n} : \Omega J + J \Omega^\ast =0 \}$ to $G$.

Differential equations of this form describe a wide range of physical systems with numerous practical applications. Some well known examples are the \emph{symplectic group} $\mathrm{Sp}_n(\C)$, with $J = \begin{pmatrix}
0 & I \\
-I & 0
\end{pmatrix}$, related to the investigations of Hamiltonian systems, and the \emph{Lorenz group} $\mathrm{SO}_{3,1}(\R) = \mathrm{O}_{3,1}(\R) \cap \mathrm{SL}_4(\R)$, with $J = \mathrm{diag}(1,1,1,-1)$, related to the study of special relativity. 
Another important example is the \emph{unitary group} $\mathrm{U}_n(\C)$, with $J=I$, which is germane to the investigation of quantum systems and has received particular attention over the last decades. Lie groups of this form were termed {\em quadratic} Lie groups in \cite{Iserles:2001}. 

In the design of numerical schemes for \eqref{eq:pb}, often a key requirement is that the approximate solution respects the conservation laws of the underlying physical systems.
However, preservation of the complex quadratic invariant of motion $AJA^\ast - J = 0$ can be challenging: If $J=I$ (unitary group), for instance, then Gauss--Legendre collocation methods are the only Runge--Kutta methods that preserve the invariant \cite{Dieci1994}. 
Rather than seeking an integrator that can preserve the conserved quantities defining $G$, which broadly defines the field of geometric numerical integration \cite{HLW:2006, BC:2017}, Lie group methods  \cite{IMNZ:2000} are a narrower class of methods that exploit the intrinsic geometric Lie group structure. By ensuring that the approximate solution evolves on the Lie group, these methods ensure that the symmetries of the system and the corresponding conservation laws are respected.

Concerns such as computational accuracy, processing time, and memory usage, which motivate algorithm development for ODEs and PDEs more broadly, also remain of paramount importance when solving ODEs and PDEs on Lie groups such as \eqref{eq:pb}. Quantum optimal control algorithms \cite{BTP1983, BCR2010, FP2021, GBCKKK2015, VRTGS2016} 
for instance,  require repeated integration of the underlying quantum differential systems, which needs to be fast and accurate while preserving the geometric properties of the system such as unitarity and conservation of energy. 
The aim of this work is to develop numerical integrators for non-autonomous linear differential equations of the form \eqref{eq:pb} which, in a context such as that of quantum optimal control, are accurate and fast, while ensuring that the numerical solution evolves on the Lie group \eqref{eq:Lie group}. 

A prominent tool in the design of Lie group methods for solving \eqref{eq:pb} has been the Magnus expansion \cite{IMNZ:2000, Magnus1954}. 
The general idea of the Magnus approach is to {write the solution as $Y(t) = \mathrm{e}^{\Om(t)}Y_0$ and expand $\Om(t)$ into an infinite series, 
\[
\Om(t) = \int_0^t A(t_1) \diff t_1 + \frac{1}{2} \int_0^t \int_0^{t_1} [A(t_1), A(t_2)] \diff t_2 \diff t_1 + \cdots,
\]
involving nested commutators of $A(t)$ at different times.
Truncating the series at a specified order and computing its exponential yields an approximation of the solution to that order.

The resulting numerical schemes, although very accurate, face many practical difficulties due to the presence of nested commutators in the Magnus expansion. The prior computation of these commutators can be very expensive, the number of commutators increases rapidly with the order of accuracy required \cite{MO1999}, they reduce sparsity \cite{ADHHJ2022} and can alter the structure sufficiently to make Magnus based scheme infeasible without substantial alterations \cite{CFBS2023}.

A very versatile technique for overcoming this difficulty is presented by the so-called  {\em commutator-free} or {\em quasi-Magnus} methods \cite{AF:2011, BM2006, CCLMO2022}.
While the derivation of these methods also starts from the Magnus expansion, they utilise the Baker-Campbell-Hausdorff (BCH) formula \cite{Hall2003, Oteo1991} to approximate the exponential of the Magnus expansion by a product of multiple exponentials, 
\begin{equation}
\label{eq:comm free form}
\mathrm{e}^{\Om(t)}\ \approx\ \mathrm{e}^{S_1}\mathrm{e}^{S_2}\ldots \mathrm{e}^{S_K},
\end{equation}
where each exponent has a much simpler form specifically, the exponents $S_k$ features no commutators, while achieving the same accuracy, see for instance \cite{AF:2011, BM2006} where  $S_k$ are obtained as linear combinations or integrals of $A(t)$. 

Another crucial bottleneck in Magnus based methods as well as their commutator-free counterparts is the computation of the matrix exponential, which can be prohibitively expensive \cite{MV}. While Krylov subspace methods lead to very efficient Magnus--Lanczos solvers for small time-steps \cite{IKP, KHK2008}, polynomial approximations do not respect the Lie group structure of \eqref{eq:pb}. Where geometric numerical integration is required, rational approximations to the exponential must be utilised instead \cite{JS2023}. 

The most well known among rational approximants to the exponential are degree $(n,n)$ (i.e. diagonal) Pad\'{e} approximants. 
The $(1,1)$ Pad\'{e} approximant, called the {\em Cayley transform}, preserves the mapping from Lie algebra to the Lie group for quadratic Lie groups of the form \eqref{eq:Lie group}, and thus is suitable for applications to \eqref{eq:pb}. It leads to the well known Crank--Nicholson method, which is a second order method. The fourth-order Magnus expansion as well as fourth-order commutator-free methods require each exponential, $\mathrm{e}^{\Om(t)}$ or $\mathrm{e}^{S_k}$ in \eqref{eq:comm free form}, to be computed with the degree $(2,2)$ Pad\'{e} approximant, while sixth-order methods need to be paired with the degree $(3,3)$ Pad\'{e} approximant. Since a degree $(n,n)$ approximant involves $n$ linear equations to solve, the requirement of high-order rational approximants leads to an $n$-fold increase in the cost of the overall scheme.

Keeping the eventual approximation of the exponential by a rational function in mind, {\em Cayley--Magnus} methods \cite{DLP:1998, Iserles:2001} develop an alternative to the Magnus series by seeking an expansion whose Cayley transform directly provides a high-order approximation to the solution of \eqref{eq:pb}. Since the Cayley transform is a degree $(1,1)$ rational method, this approach circumvents the $n$-fold scaling of traditional Magnus based approaches. However, much like the Magnus expansion, this new expansion called the {\em Cayley--Magnus expansion} also features commutators, and its application involves very similar challenges due to their presence. 
To the best of our knowledge, there is no commutator-free alternative for the Cayley--Magnus methods.

In this work we propose a new approach which combines the commutator-free approach with Cayley--Magnus expansion to derive high-order schemes which avoid both, nested commutators and exponential matrix computations. The resulting schemes have close parallels to \eqref{eq:comm free form}, with the exponentials being replaced by the significantly cheaper Cayley transforms. The schemes respect the Lie group structure of \eqref{eq:pb} for quadratic Lie groups of the form \eqref{eq:Lie group} by design. 

The rest of the article is organized as follows. In Section~\ref{sec:Cayley-transform} we introduce some notations and definitions. Moreover, we recall results on Cayley--Magnus and Legendre expansions that we will need to build a fourth-order scheme based on the Cayley transform. Section~\ref{sec:Cayley-commutator-free} is the core of the paper: We first present a variant of the Cayley--BCH formula derived up to order four; the Cayley-BCH expansion is then used to derive a new fourth-order commutator-free Lie group integrator for quadratic Lie groups such as $\mathrm{SO}_n(\R)$ or $\mathrm{U}_n(\C)$. Section~\ref{sec:Examples} contains numerical experiments that demonstrate the effectiveness of the proposed approach.

%---------------------------------------
%---------------------------------------
\section{Preparation}
\label{sec:Cayley-transform}
%---------------------------------------
%
%---------------------------------------
\subsection{Cayley transforms for quadratic Lie groups}
Consider the matrix differential equation
\begin{equation}
\label{eq:initial_system}
\dot{Y}(t) = A(t)Y(t), \quad Y(t_0) = Y_0, \quad Y(t) \in G, \quad t \in [t_0, t_f],
\end{equation}
where $A(\cdot)$ is a Lipschitz-continuous operator taking its values in $\mathfrak g$.
If $G$ is a Lie group and $\mathfrak g$ its Lie algebra, then the motions evolves on $G$, i.e.\ $Y(t) \in G$ for all $t$, provided that $A(\cdot)$ takes values in $\mathfrak g$ and $Y_0 \in G$. 
The article focuses on numerical methods for differential equations on {\em quadratic Lie groups} which are matrix Lie groups of the form
\begin{equation}
\label{eq:Lie-group}
G = \{M \in \mathrm{GL}_n(\C) \, : \, MJM^* = J\},
\end{equation}
for an invertible matrix $J \in \mathrm{GL}_n(\C)$. Here $M^* = \bar{M}^\top$ denotes the conjugate transpose.
The Lie algebra to $G$ is given as
\[
\mathfrak g = \{ \Omega \in \C^{n \times n} \, : \, \Omega J + J \Omega^\ast =0 \}.
\]

An important example of \eqref{eq:Lie-group} is the unitary group $\mathrm{U}_n(\C)$, with $J = I$ (where $I$ is the identity matrix), which occurs in the context of quantum systems. Its Lie algebra consists of skew-hermitian matrices. Another example is the (complex) symplectic group $\mathrm{Sp}_{2n}(\C)$, with $J=\begin{pmatrix}
0 & I \\
-I & 0
\end{pmatrix}$. Its Lie algebra consists of Hamiltonian matrices. 

Let $c \in \C \setminus \{0\}$. For $\Om \in \C^{n \times n}$ with $c^{-1} \notin \sigma(\Om)$ the $c$-Cayley transform of $\Om$ is defined as
\[
\Cay(\Om, c) = \left(I - c \Om\right)^{-1}  \left(I + c^\ast \Om \right).
\]
Here $\sigma(M)$ denotes the spectrum of $M$. 
The inverse $c$-Cayley transform is given as
\[
\mathrm{Cay}^{-1}(A,c) = -\frac 1 {c^\ast} \left(I+\frac{c}{c^\ast}A\right)^{-1}(I-A) \quad \text{for } -\frac{c^\ast}{c} \not \in \sigma(A).
\]
Indeed, the $c$-Cayley transform constitutes a diffeomorphism $\Cay_{c} \colon \tilde{\mathfrak g} \stackrel{\sim}{\longrightarrow} \widetilde G$
between $\tilde{\mathfrak g} = \{ \Omega \in \mathfrak g  : c^{-1} \notin \sigma(\Omega) \}$ and $\widetilde G =  \{ A \in G : -\frac{c^\ast}{c} \notin \sigma (A) \}$. 

Since Cayley transforms respect the Lie group structure for quadratic Lie groups, so do their products, i.e. the product $\Cay(\Om,c_1)\Cay(\Om,c_2)\ldots\Cay(\Om,c_n)$ resides in the Lie group $\widetilde{G} = \{ A \in G : -\frac{c_k^\ast}{c_k} \notin \sigma (A),\, k=1,\ldots,n\}$, provided $\Om \in \widetilde{\mathfrak g} = \{ \Omega \in \mathfrak g  : c_k^{-1} \notin \sigma(\Omega),\ k=1,\ldots,n \}$. Thus Cayley transforms are natural building blocks for rational approximations that respect quadratic Lie groups. Indeed, all unitary rational approximations (relevant to quantum dynamics, where $G=\mathrm{U}_n(\C)$), including higher-order diagonal Pad\'{e} approximations, can be obtained as compositions of Cayley transforms \cite{JS2023}.

%, making Cayley transforms the building blocks of rational approximations that respect the Lie group structure of the differential equation \eqref{eq:pb}.

We will use the Cayley transform as a cheap alternative to the surjective matrix exponential $\exp \colon \mathfrak g \rightarrow G$ to design Lie group integrators. 
Other approaches employing the Cayley transform to solve system \eqref{eq:initial_system} can be found for example in \cite{Iserles:2001, MB:2001}. These approaches are generally based on the following result.

In the following, the $1/2$-Cayley transform will simply be referred to as {\em Cayley transform} and will be denoted by $\Cay(A)$.
\begin{lemma}
\label{lemma:1}
Let $Y(t)$ be the solution of system \eqref{eq:initial_system}, with $-1 \notin \sigma(Y(t)Y_0^{-1})$ for any $t \in [t_0,t_f]$,
then $Y(t)$ can be written in the form $Y(t) = \Cay(\Om(t))Y_0$, where the matrix $\Om \in \tilde{\mathfrak g}$ satisfies 
\begin{equation}
\label{eq:skew-Herm-prob}
\dot{\Om}(t) = A(t) - \frac{1}{2}[\Om, A(t)] - \frac{1}{4}\Om A(t) \Om, \quad \Om(t_0) = \Om_0, 
\end{equation}
with the Lie bracket (commutator) of two matrices $A$ and $B$ defined by $[A,B] = A\cdot B - B\cdot A$.
\end{lemma}

\begin{proof}
To simplify the notations and without loosing any generality, we consider $Y_0 = I$. 
%\PS{[What is $I_d$?? Shouldn't this be $I$?]}. 
So,
$Y = \Cay(\Om) ~ \Longrightarrow ~ \left(I - \frac{\Om}{2}\right) Y = \left(I + \frac{\Om}{2}\right)$. Differentiation of this relation leads to
\begin{equation}
\label{eq:proof-lemma-11}
\begin{aligned}
- \frac{\dot\Om}{2} Y + \left(I - \frac{\Om}{2}\right) \dot Y  = \frac{\dot\Om}{2} \quad
\text{i.e} \quad \dot \Om &= 2 \left(I - \frac{\Om}{2} \right) \dot Y  \left(I + Y\right)^{-1} \\
\text{i.e} \quad \dot \Omega &= 2 \left(I - \frac{\Om}{2}\right) A(t)Y \left(I + Y \right)^{-1} 
\end{aligned}
\end{equation}
Moreover,
\begin{equation}
\label{eq:proof-lemma-12}
\begin{aligned}
\left(I + Y \right) Y^{-1} &= \left(I +  \left(I - \frac{\Om}{2} \right)^{-1}  \left(I + \frac{\Om}{2} \right) \right)\left(I + \frac{\Om}{2} \right)^{-1}  \left(I - \frac{\Om}{2} \right)\\
&= \left(I + \frac{\Om}{2} \right)^{-1} \left(I - \frac{\Om}{2} \right) + I \\
&= \left(I + \frac{\Om}{2} \right)^{-1}\left(I - \frac{\Om}{2} + I + \frac{\Om}{2} \right) = 2 \left(I + \frac{\Om}{2} \right)^{-1}
\end{aligned}
\end{equation}
Plugging in relation~\eqref{eq:proof-lemma-12} into equation~\eqref{eq:proof-lemma-11}, one obtains
$
\dot\Omega = \left(I - \frac{\Om}{2}\right) A(t) \left(I + \frac{\Om}{2} \right) 
$
which allows us to conclude \eqref{eq:skew-Herm-prob}.
\end{proof}

\begin{remark}
%\boris{done}
In numerical time-stepping methods, the time step $\delta t$ can always be made sufficiently small such that $\Omega(t)$ is close to the zero matrix such that $\Omega(t) \in \tilde{\mathfrak{g}}(\mathbb C)$, i.e.\ $2 \not \in \sigma(\Omega(t))$ for all $t \in [t_0,t_0 + \delta t]$. While this could force small time-steps in a general setting,
in the context of quantum systems, one has $\tilde{\mathfrak g} = \tilde{\mathfrak u}_n(\C) = \mathfrak u_n(\C) = \mathfrak g$ since the solution evolves on the unitary matrix group, so that the condition $\Om \in \tilde{\mathfrak u}_n(\C)$ is automatically satisfied.

\end{remark}
%\boris{done!}
According to Lemma~\ref{lemma:1}, solving system~\eqref{eq:initial_system} is equivalent to solving system~\eqref{eq:skew-Herm-prob}, but now considering time-stepping with a small time step $\delta t$ in order to guarantee the existence of $\Om \in \tilde{\mathfrak g}$.
As the Lie algebra is characterised by the linear constraint $AJ+JA^\ast =0$, we can apply any Runge--Kutta method to~\eqref{eq:skew-Herm-prob} and obtain a Lie group structure preserving integrator, since Runge--Kutta methods preserve linear constraints. This approach is suggested in \cite[IV.8.3]{HLW:2006}, for instance. However, this involves the repeated computation of matrix commutators, which can be costly in high dimensions. 
So next, we focus on system \eqref{eq:skew-Herm-prob} and present the Magnus expansion for Cayley transform, developed by Iserles \cite{Iserles:2001}, which will be used later to derive our methods.

%---------------------------------------
\subsection{The Cayley--Magnus expansion}

We now focus on system \eqref{eq:skew-Herm-prob}
given in Lemma~\ref{lemma:1} i.e
\begin{equation}
\tag{3}
\dot{\Om}(t) = A(t) - \frac{1}{2}[\Om, A(t)] - \frac{1}{4}\Om A(t) \Om, \quad \Om(t_0) = \Om_0. 
\end{equation}
In this section we recall the results of \cite{Iserles:2001} for expanding $\Om$ as a Cayley--Magnus series. For this, let $\Omega(t)$ denote the solution of \eqref{eq:skew-Herm-prob} to the initial value $\Omega(0) =0$. We consider as an ansatz the formal series
\begin{equation}
\label{eq:Cayley-expension}
\Om(t) = \sum_{m=1}^\infty \Om_m(t),
\end{equation}
where $\Omega_m(t)$ denotes an expression consisting of $m$ iterated integrals over polynomials of degree $m$ in $A$.
Substituting this into \eqref{eq:skew-Herm-prob} and integrating over $[0,t]$ with $\Omega(0)=0$ leads to
\[
\begin{aligned}
\sum_{m=1}^\infty  \Omega_m  &= \int_0^t A(\xi) \diff \xi - \frac{1}{2} \int_0^t \left[ \sum_{m=1}^\infty \Om_m(\xi), A(\xi) \right] \diff \xi \\
&\hspace{4cm} - \frac{1}{4} \int_0^t \left[ \sum_{m=1}^\infty \Om_m(\xi) \right] A(\xi) \left[ \sum_{m=1}^\infty \Om_m(\xi) \right] \diff \xi \\
&=  \int_0^t A(\xi) \diff \xi - \frac{1}{2} \int_0^t \sum_{m=2}^\infty [ \Om_{m-1}(\xi), A(\xi) ] \diff \xi \\
&\hspace{4cm} - \frac{1}{4} \sum_{m=3}^\infty \sum_{k=1}^{m-2} \int_0^t \Om_{m-k-1}(\xi)  A(\xi) \Om_k(\xi) \diff \xi \\
%&=  \int_0^t A(\xi) \diff \xi - \frac{1}{2} \int_0^t [A(\xi), A(\xi)] \diff \xi \\
&=  \int_0^t A(\xi) \diff \xi - \frac{1}{2} \int_0^t [\Omega_1(\xi), A(\xi)] \diff \xi \\
&~~~~~~ - \sum_{m=3}^\infty \left[\frac{1}{2} \int_0^t \left[\Om_{m-1}(\xi), A(\xi)\right] \diff \xi - \frac{1}{4} \sum_{k=1}^{m-2} \int_0^t \Om_{m-k-1}(\xi)  A(\xi) \Om_k(\xi) \diff \xi \right].
\end{aligned}
\]
Now $\Omega_j$ can be determined recursively as follows:
\begin{equation*}
\begin{aligned}
\Om_1(t) &= \int_0^t A(\xi) \diff \xi,\\
%\Om_2(t) &= - \frac{1}{2} \int_0^t [A(\xi), A(\xi)] \diff \xi\\
\Om_2(t) &= - \frac{1}{2} \int_0^t [\Omega_1(\xi), A(\xi)] \diff \xi,\\
%, ~~ \text{and for} ~~ m \geq 3,
\Om_m(t) &= \frac{1}{2} \int_0^t [\Om_{m-1}(\xi), A(\xi)] \diff \xi - \frac{1}{4} \sum_{k=1}^{m-2} \int_0^t \Om_{m-k-1}(\xi)  A(\xi) \Om_k(\xi) \diff \xi, ~ \text{for} ~ m \geq 3.
\end{aligned}
\end{equation*}
For a combinatorical description of the expressions $\Omega_m$ using the language of trees and for a discussion of convergence properties of the series $\sum_{m=1}^\infty \Omega_m$ we refer to \cite{Iserles:2001}.
The first terms of this expansion are explicitly given by
\[
\begin{aligned}
\Om_1(t) &= \int_0^t A(\xi) \diff \xi, \quad \quad \quad 
\Om_2(t) = - \frac{1}{2} \int_0^t \int_0^{\xi_1} [A(\xi_2), A(\xi_1)] \diff \xi_2 \diff \xi_1,  \\
\Om_3(t) &= \frac{1}{4} \int_0^t \int_0^{\xi_1} \int_0^{\xi_2} \left[ [A(\xi_3), A(\xi_2)], A(\xi_1)\right] \diff \xi_3 \diff \xi_2 \diff \xi_1 \\
&\quad\quad - \frac{1}{4} \int_0^t \left[\int_0^{\xi_1} A(\xi_2) \diff \xi_2 \right] A(\xi_1) \left[\int_0^{\xi_1} A(\xi_2) \diff \xi_2 \right] \diff \xi_1.
\end{aligned}
\]

\begin{lemma}
Truncating the Cayley--Magnus expansion \eqref{eq:Cayley-expension} at a given order $p$, i.e.~setting
\begin{equation}
\label{eq:Cayley-truncate-expension}
\Om(t) \approx \Om_p(t) = \sum_{m=1}^{p} \Om_p,
\end{equation}
leads to a $p$-order approximation.
\end{lemma}

\begin{remark}
Only a few integral terms in each $\Om_m$ are relevant to obtain a $p$-order approximation. Thus, by considering only relevant terms, we can considerably reduce the number of terms to be computed (there are fewer terms than in the exponential Magnus expansion, see again \cite{Iserles:2001} for more explanation). Moreover, an approximation of order-p for $\Om(t)$ will lead to an approximation of order-p for $Y(t)$, i.e. $Y(t) = \Cay(\Om_p(t))Y_0 + O(t^{p+1})$ (see for instance \cite{DLP:1998}). 
\end{remark}

%---------------------------------------
\subsection{Cayley--Magnus and Legendre expansion}

A starting point of deriving a commutator-free higher-order scheme is to consider a Legendre expansion of $A(\cdot)$, following the strategy in \cite{AF:2011}. We will first introduce the Legendre expansion of the matrix $A$ and expand $\Om$ in terms of the Legendre expansion. The shifted Legendre polynomials $P_n(x)$ are defined for $n=0,1,2,\cdots$ through the recurrence
\begin{equation}
\label{eq:Legendre-polynomial}
P_0(x) = 1, ~~ P_1(x) = 2x-1, ~~ P_{n+1}(x) = \frac{2n+1}{n+1} P_n(x) - \frac{n}{n+1} P_{n-1}(x). 
\end{equation}
By definition, $P_n(x)$ is a polynomial of degree $n$, symmetric with respect to $1/2$. The first terms are explicitly given by
\small{
\begin{equation*}
P_2(x) = 6x^2 - 6x + 1, ~~ P_3(x) = 20x^3 - 30x^2 + 12x - 1, ~~ P_4(x) = 70x^4 - 140x^3 + 90x^2 - 20x + 1.
\end{equation*}}
For a given time step $\delta t$, the matrix $A$ can be expanded on the interval $[0, \delta t]$ in a series of Legendre polynomials given by (see also \cite{AF:2011})
\begin{equation}
\label{eq:Legendre-transform-of-A}
A(t) = \frac{1}{\delta t} \sum_{k=1}^N A_k P_{k-1}\left(\frac{t}{\delta t}\right) + O(\delta t^{N+1}), \quad (0\leq t \leq \delta t)
\end{equation}
where $P_k$, $k=0,1,2,\ldots$ are the Legendre polynomials defined in \eqref{eq:Legendre-polynomial} and where the coefficients $A_k$ are given by
\[
A_k = (2k-1) \int_0^{\delta t} A(t) P_{n-1}\left(\frac{t}{\delta t} \right) \diff t = (2k-1)\delta t \int_0^1 A(x\delta t) P_{n-1}(x) \diff x.
\]
Plugging the Legendre expansion \eqref{eq:Legendre-transform-of-A} into the Cayley Magnus expansion, one can express $\Om(t)$ with respect to the coefficients $A_k$. The first three terms read %explicitly
\[
\begin{aligned}
\Om_1(\delta t) &= \frac{1}{\delta t} \sum_{n=1}^N \int_0^{\delta t} A_n P_{n-1}\left(\frac{\xi}{\delta t}\right) \diff \xi + O(\delta t ^{N+1}) \\
&= \sum_{n=1}^N \left\{\int_0^{1} P_{n-1}(x) \diff x \right\} A_n + O(\delta t ^{N+1}) \\
&= A_1, ~~ \text{since $P_k$ ($k\geq 1$) is anti-symmetric w.r.t $1/2$ and} ~ \int_0^{1} P_{0}(x) \diff x = 1, \\
\end{aligned}
\]
\[
\begin{aligned}
\Om_2(\delta t) &= -\frac{1}{2} \int_0^1 \int_0^{\xi_1} [A(\xi_2), A(\xi_1)] \diff \xi_2 \diff \xi_1 \\
&= -\frac{1}{2 \delta t^2} \int_0^1 \int_0^{\xi_1} \left[\sum_{n=1}^N A_n P_{n-1}\left(\frac{\xi_2}{\delta t}\right), \sum_{k=1}^N A_k P_{k-1}\left(\frac{\xi_1}{\delta t}\right) \right] \diff \xi_2 \diff \xi_1 \\
&= -\frac{1}{2} \sum_{n,k=1}^N \left\{ \int_0^{1} \int_0^{x_1} P_{n-1}(x_2)P_{k-1}(x_1) \diff x_2 \diff x_1 \right\} [A_n, A_k] + O(\delta t ^{N+1}), \\
\end{aligned}
\]
\[
\begin{aligned}
\Om_3(\delta t) &= \frac{1}{4} \int_0^t \int_0^{\xi_1} \int_0^{\xi_2} \left[ [A(\xi_3), A(\xi_2)], A(\xi_1)\right] \diff \xi_3 \diff \xi_2 \diff \xi_1 \\
&\quad\quad - \frac{1}{4} \int_0^{\delta t} \left\{\int_0^{\xi_1} A(\xi_2) \diff \xi_2 \right\} A(\xi_1) \left\{\int_0^{\xi_1} A(\xi_2) \diff \xi_2 \right\} \diff \xi_1 \\
&= \frac{1}{4} \sum_{n,m,k=1}^N \left\{ \int_0^{1} \int_0^{x_1} \int_0^{x_2} P_{n-1}(x_3)P_{m-1}(x_2)P_{k-1}(x_1) \diff x_3 \diff x_2 \diff x_1 \right\} \cdot \\ &\hspace{8.5cm} \left[ [A_n, A_m], A_k \right] \\
& - \frac{1}{4} \sum_{n,m,k=1}^N \left( \int_0^1 \left\{\int_0^{x_1} P_{n-1}(x_2) \diff x_2 \right\} P_{m-1}(x_1) \left\{\int_0^{x_1} P_{k-1}(x_2) \diff x_2 \right\} \diff x_1 \right) \cdot \\
& \hspace{8cm} (A_n A_m A_k) + O(\delta t ^{N+1}). \\
\end{aligned}
\]

Next, we will denote by $\Om^{[N]}_k(\delta t)$ the truncation of $\Om_k(\delta t)$ up to the first $N$ terms of the Legendre coefficients.

\begin{proposition}
\label{prop:1}
For a given time-step $\delta t$, the first three terms of the Cayley--Magnus expansion combined with a Legendre expansion of $A$ until $N=3$ are given by
\begin{equation*}
\label{eq:eaxct-expansion}
\begin{aligned}
\Om_1(\delta t) &= \Om_1^{[2]}(\delta t) = A_1, \quad\quad \Om_2^{[2]}(\delta t) =-\frac{1}{6} [A_1, A_2], \\
\Om_3^{[2]}(\delta t) &= - \frac{1}{12} A_1^3 - \frac{1}{120} A_2^3 + \frac{1}{60} A_1A_2^2 - \frac{1}{30} A_2A_1A_2 + \frac{1}{60} A_2^2A_1. 
\end{aligned}
\end{equation*}

Moreover, we have
\begin{equation}
\label{eq:true-approx}
\Om(\delta t) = \Om_1(\delta t) + \Om_2^{[2]}(\delta t) + \Om_3^{[2]}(\delta t) + O(\delta t^4).
\end{equation}
\end{proposition}

\begin{remark}
\label{rmk:2}
For any $n \in \mathbb{N}^*$, $A_n$ is a term of order $\delta t^n$. This can be easily seen by comparing the Legendre expansion with an expansion $A(t) = \sum_{m\geq 1} a_m t^{m-1}$ in powers of $t$. Since we are looking to build a fourth-order scheme, according to the expression of $\Om_1$, $\Om_2$ and $\Om_3$, only the first two terms of this expansion, i.e.~$A_1$ and $A_2$ will be relevant for us.
\end{remark}

The approximation in \eqref{eq:true-approx} is a third-order approximation when considering the exact integral to compute $A_1$ and $A_2$. However by a matter of fact, the order is increased to fourth when we consider specifically the Gauss-Legendre quadrature to approximate this integral, i.e.\ the quadrature error exactly cancels the leading error term of the Cayley expansion. This later point has been discussed by Iserles \cite{Iserles:2001}. Thus, taking
\[
A^1 = A\left(t_n + \left(\frac{1}{2}-\frac{\sqrt{3}}{6} \right) \delta t\right), \quad A^2 = A\left(t_n + \left(\frac{1}{2}+\frac{\sqrt{3}}{6} \right) \delta t \right),
\]
and
\begin{equation}
\label{eq:Gauss-Legendre}
A_1 = \frac{\delta t}{2}(A^1 + A^2) + O(\delta t^5), \quad  A_2 = \frac{\delta t\sqrt{3}}{2}(A^2 - A^1) + O(\delta t^5),
\end{equation}
one has the following result.

\begin{proposition}
\label{prop:2}
Given a time step $\delta t$, the following approximation holds,
\begin{equation}
\label{eq:simplify-approx}
\Om(t_n + \delta t) = A_1 -\frac{1}{6} [A_1, A_2] - \frac{1}{12} A_1^3  + O(\delta t^5),
\end{equation}
with $A_1$ and $A_2$ defined in \eqref{eq:Gauss-Legendre}.
\end{proposition}

\begin{proof}
From Proposition~\ref{prop:1}, we have 
\[
\Om(\delta t) = \Om_1(\delta t) + \Om_2(\delta t) + \Om_3(\delta t) + O(\delta t^4).
\]
However, apart from $A_1^3$, all the other terms in $\Om_3$ are of order greater than $\delta t^5$, since $A_k$ is of order $\delta t^k$ (see Remark~\ref{rmk:2}), so we first obtain
\[
\Om(\delta t) = A_1(\delta t) -\frac{1}{6} [A_1(\delta t), A_2(\delta t)] - \frac{1}{12} A_1^3(\delta t)  + O(\delta t^4).
\]
Approximating $A_1$ and $A_2$ by the Gauss-Legendre quadrature as defined in \eqref{eq:Gauss-Legendre}, $\Om_3$ becomes symmetric so that the order of the error will be even, see \cite[Sec.~4.2]{Iserles:2001}. Thus the order increases to $\delta t^5$. 
\end{proof}

\begin{remark}
\label{remark:quadrature}
In the previous formulation, the use of Gauss quadrature is not strictly necessary, as other quadrature rules can also yield a fourth-order scheme. However, any alternative quadrature must be symmetric and of at least fourth order. An example illustrating this, based on a quadrature with nodes at $\delta t/4$, $\delta t/2$, and $3\delta t/4$, is provided in \ref{App2}.  
\end{remark}

Using the Gauss-Legendre quadrature given above to compute $A_1$ and $A_2$, the Cayley Magnus Time-propagator (CMT) scheme \eqref{eq:simplify-approx} is already a fourth-order approximation scheme as desired (see also  \cite{Iserles:2001}). However, it still contains commutators which can make the implementation expensive for large systems, can reduce sparsity, and may be more structurally complicated. The aim in the following section is to use a similar idea as the exponential commutator-free method developed in \cite{AF:2011}, using in our case the Cayley transform. To this end, we require a BCH-type formula for the Cayley transform.

%---------------------------------------
%---------------------------------------
\section{Commutator-free Cayley scheme}
\label{sec:Cayley-commutator-free}
%---------------------------------------
%---------------------------------------

%---------------------------------------
\subsection{BCH-formula for Cayley transform}

For the derivation of the fourth-order commutator-free Cayley methods, we need to express the product of three Cayley transforms as a single Cayley transform up to order four accuracy. While this can be obtained by two applications of the Cayley-BCH formula developed by Iserles and Zanna \cite{IZ:2000}, for the sake of concreteness we show the derivation explicitly up to order four. 

\begin{proposition}
\label{prop:Cayley-BCH}

Let $A, B, C \in \mathfrak g$ be three matrices in a neighborhood of $0$, then the following formula holds
\begin{equation}
\label{eq:3-Cayley-composition}
\Cay(A) \Cay(B) \Cay(C) = \Cay(\Om(A,B,C)), 
\end{equation}
with                               
\begin{equation}
\label{eq:composition_matrix-3}
\begin{aligned}
\Om(A,B,C) =  A + B + C + \frac{1}{2}\left([A,B] + [A,C] + [B,C] \right) + \frac{1}{4}[A,B],C]  & \\
- \frac{1}{4}(ACB + BCA) - \frac{1}{4}\left(ABA + ACA + BAB + BCB \right. & \\
\left. + CAC + CBC \right) + F(A,B,C), &
\end{aligned}
\end{equation}
where $F$ is a series of homogeneous polynomials in $A$, $B$, $C$ of degrees $m$ with $m \ge 4$.
\end{proposition}

\begin{proof}
First, let us observe that for a small enough neighborhood of $0$, one has $1/2 \notin \sigma(A) \cup \sigma(B) \cup \sigma(C)$. 
Now, considering that $\Cay(A) \Cay(B) \Cay(C) = \Cay(\Om(A,B,C)) $, then we get
\[
\begin{aligned}
& \Cay(A) \Cay(B) \Cay(C) = \left( I - \frac{\Om(A,B,C)}{2} \right)^{-1}  \left(I + \frac{\Om(A,B,C)}{2} \right), \\
\Rightarrow ~~& \left( I - \frac{\Om(A,B,C)}{2} \right) \Cay(A) \Cay(B) \Cay(C) = \left(I + \frac{\Om(A,B,C)}{2} \right) \\
\Rightarrow ~~& \frac{\Om(A,B,C)}{2}\left( I + \Cay(A) \Cay(B) \Cay(C)\right) = \Cay(A) \Cay(B) \Cay(C) - I \\
\Rightarrow ~~ & \Om(A,B,C) = -2 \left(I - \left(I - \frac{A}{2} \right)^{-1} \left(I + \frac{A}{2} \right)\left(I - \frac{B}{2} \right)^{-1} \left(I + \frac{B}{2} \right) \cdot \right. \\
&\hspace{2.5cm} \left. \left(I - \frac{C}{2} \right)^{-1} \left(I + \frac{C}{2} \right) \right)  \cdot \left(I + \left(I - \frac{A}{2} \right)^{-1} \left(I + \frac{A}{2} \right) \cdot \right. \\
&\hspace{2.5cm} \left. \left(I - \frac{B}{2} \right)^{-1} \left(I + \frac{B}{2} \right) \left(I - \frac{C}{2} \right)^{-1} \left(I + \frac{C}{2} \right) \right)^{-1}.
\end{aligned}
\]
A series expansion\footnote{Details of the computation can be found in the appendix.} of this last relation leads to the desired result.
\end{proof}

\begin{remark}
If we want a fourth-order scheme, the terms of the formula obtained in Proposition~\ref{prop:Cayley-BCH} by ignoring $F(A,B,C)$ are enough. However, one needs to compute more terms if we are looking at more than fourth-order. 
\end{remark}

The Cayley transform version of the usual BCH- and sBCH-formula \cite{IZ:2000} can then be deduced from Proposition~\ref{prop:Cayley-BCH} by setting $C = 0$ and $C=A$ respectively. 

\begin{corollary}[BCH-formula]
Considering $A, B \in \mathfrak{g}$ in the neighborhood of $0$, then one has
\begin{equation}
\label{eq:2_Cayley_composition}
\Cay(A) \Cay(B) = \Cay(\Om(A,B)),
\end{equation}
with 
\begin{equation}
\label{eq:composition_matrix-2}
\begin{aligned}
\Om =  A + B + \frac{1}{2}[A,B] - \frac{1}{4}(ABA + BAB) + F(A,B)
\end{aligned}
\end{equation}
where $F$ is a
series of homogeneous polynomials in $A$, $B$ of degrees $m$ with $m \ge 4$.
\end{corollary}

\begin{remark}
If we consider the general Cayley transform, then 
\[
\Cay_{c_1}(A) \Cay_{c_2}(B) = \Cay(\Om)
\]
leads to
\[
\begin{aligned}
\Om =  \frac{2 x_1}{|c_1|^2} A + \frac{2 x_2}{|c_2|^2} B + \frac{2x_1x_2}{|c_1|^2|c_2|^2}[A,B] - \frac{2 x_1x_2}{|c_1|^2|c_2|^2} \left(\frac{x_1}{|c_1|^2} ABA + \frac{x_2}{|c_2|^2} BAB \right)&  \\
\hspace{0.2cm}  + 2i \left( \frac{y_1}{c_1|c_1|^4} A^3  + \frac{y_2}{c_2|c_2|^4} B^3 + \frac{x_1 x_2 y_2}{|c_1|^2|c_2|^4} B^2 A - \frac{x_1 x_2 y_1}{|c_1|^4|c_2|^2} A^2 B - \frac{2 x_1 x_2 y_2}{|c_1|^2|c_2|^4} A B^2 \right. & \\
\left. - \frac{y_1}{|c_1|^4} A^2 - \frac{y_2}{|c_2|^4} B^2 \right) +  F(A,B), &
\end{aligned}
\]
with $c_1 = x_1 + i y_1$, $c_2 = x_2 + i y_2$ and where $F(A,B)$ is as in the previous corollary.
\end{remark}

\begin{corollary}[sBCH-formula]
Considering $A, B \in \mathfrak g$ in the neighborhood of $0$, then one has
\begin{equation}
\label{eq:2-Cayley-composition}
\Cay(A) \Cay(B) \Cay(A) = \Cay(\Om(A,B)) 
\end{equation}
where 
\begin{equation}
\begin{aligned}
\Om =  2A + B - \frac{1}{2}(A^2B + BA^2 + BAB + A^3) + F(A,B)
\end{aligned}
\label{eq:composition_matrix_2}
\end{equation}
where $F$ is a series of homogeneous polynomials in $A$ and $B$ of degree $m$ with $m \ge 4$.
\end{corollary}

%---------------------------------------
\subsection{Commutator-Free Cayley Time-propagator scheme (CFCT)} 

In this section, in analogy to the commutator-free quasi-Magnus integrators, we seek a fourth-order approximation of the form
\begin{equation}
\label{eq:com-free-scheme}
Y(\delta t)\ \approx Y_1 := \Cay(\Om_1(\delta t)) \Cay(\Om_2(\delta t)) \Cay(\Om_3(\delta t))Y_0 = \Cay(\widetilde \Om(\delta t))Y_0,
\end{equation}
with $\Om_i = \sum_{k=1}^2 \alpha_{i,k} A_k$, $i=1,2,3$, with $A_k$ being the Gauss-Legendre coefficients of $A$ defined in \eqref{eq:Gauss-Legendre} and $\alpha_{i,k} \in \R$.
Recall that the Lie algebra $\mathfrak g$ has the structure of a real vector space. If $A \in \mathfrak g$ then the Gauss-Legendre coefficients are elements of $\mathfrak g$. Since we are seeking a Lie group structure-preserving scheme, we seek real coefficients $\alpha_{i,k}$ (rather than complex coefficients) such that $\Omega_i \in \mathfrak g$ is guaranteed.
We want to find, if there exists, the correct coefficients $ \alpha_{i,k}$, $k=1,2$ and $i=1,2,3$, such that approximation \eqref{eq:com-free-scheme} leads to a fourth-order approximation.
Using the Cayley--BCH formula from Proposition~\ref{prop:Cayley-BCH}, we get
\begin{eqnarray}
\label{eq:BCH-develop}
\widetilde \Om &=& \sum_{i=1}^3\Om_i - \frac{1}{4}(\Om_1\Om_2\Om_1 + \Om_1\Om_3\Om_1 + \Om_2\Om_1\Om_2 + \Om_2\Om_3\Om_2 + \Om_3\Om_1\Om_3 + \Om_3\Om_2\Om_3) \nonumber \\
&&~ - \frac{1}{4}(\Om_1\Om_3\Om_2 + \Om_2\Om_3\Om_1) + \frac{1}{2}([\Om_1, \Om_2] + [\Om_1, \Om_3] + [\Om_2, \Om_3]) + \frac{1}{4}[[\Om_1, \Om_2], \Om_3] \nonumber\\
&=& \sum_{i=1}^3\sum_{j=1}^2 \alpha_{ij} A_j + \frac{1}{2}\sum_{i,j=1}^2 (\alpha_{1i}\alpha_{2j} + \alpha_{1i}\alpha_{3j} + \alpha_{2i}\alpha_{3j})[A_i,A_j] - \frac{1}{4}\sum_{i,j,k=1}^2 \left(\alpha_{1i}\alpha_{3j}\alpha_{2k} \right. \nonumber \\
&&~  + \alpha_{2i}\alpha_{3j}\alpha_{1k} + \alpha_{1i}\alpha_{2j}\alpha_{1k} + \alpha_{1i}\alpha_{3j}\alpha_{1k} + \alpha_{2i}\alpha_{1j}\alpha_{2k} + \alpha_{2i}\alpha_{3j}\alpha_{2k} + \alpha_{3i}\alpha_{1j}\alpha_{3k}  \nonumber \\
&&~ + \left. \alpha_{3i}\alpha_{2j}\alpha_{3k} + \alpha_{2i}\alpha_{1j}\alpha_{3k} + \alpha_{3i}\alpha_{1j}\alpha_{2k} - \alpha_{1i}\alpha_{2j}\alpha_{3k} - \alpha_{3i}\alpha_{2j}\alpha_{1k} \right)A_i A_j A_k \nonumber 
\end{eqnarray}
On the other hand, from \eqref{eq:simplify-approx} one has
\[
\Om(\delta t) = A_1(\delta t) -\frac{1}{6} [A_1(\delta t), A_2(\delta t)] - \frac{1}{12} A_1^3(\delta t) + O(\delta t^5).
\]
Equating the two expressions for $\Omega(\delta t)$ and $\tilde\Om$, we obtain nonlinear relations for the coefficients. A solution is provided by 
\begin{equation}
\label{eq:coefs}
\begin{aligned}
\alpha_{31} &= \alpha_{11}, ~~ \alpha_{21} = 1 - 2\alpha_{11}, ~~ \alpha_{12} = - \alpha_{32} = \alpha_{11} - \alpha_{11}^2, \\
\alpha_{22} &= 0, ~~ \text{with} ~~ \alpha_{11}  = \frac{2^{1/3}}{3} +  \frac{2^{2/3}}{6} + \frac{2}{3}.
\end{aligned}
\end{equation}
We obtain the following final scheme
\begin{equation}
\label{eq:final-scheme}
\begin{aligned}
Y_1= \Cay(\alpha_{11}A_1(\delta t) + \alpha_{12}A_2(\delta t)) \Cay(\alpha_{21}A_1(\delta t)) \cdot & \\
\Cay(\alpha_{11}A_1(\delta t) - \alpha_{12}A_2(\delta t))Y_0 &
\end{aligned}
\end{equation}

\begin{proposition}
For a given time-step $\delta t$, we have
\begin{equation*}
\begin{aligned}
Y(\delta t) = \Cay(\alpha_{11}A_1(\delta t) + \alpha_{12}A_2(\delta t)) \Cay(\alpha_{21}A_1(\delta t)) \cdot & \\
\Cay(\alpha_{11}A_1(\delta t) - \alpha_{12}A_2(\delta t))Y_0 + O(\delta t^5) &
\end{aligned}
\end{equation*}
with coefficients $\alpha_{ij}$ as in \eqref{eq:coefs}.
\end{proposition}

\begin{remark}
To simplify notation, we have considered only the interval $[t_0, t_0+\delta t]$. However, since the operator $A$ is non-autonomous, the coefficients $A_k$ will naturally depend on $t_n$ for a given discretization $n = 1,\cdots,N$. So, the final scheme will look like
\[
\begin{aligned}
Y_{n+1} = \Cay(\alpha_{11}A_1(t_n,\delta t) + \alpha_{12}A_2(t_n,\delta t)) \Cay(\alpha_{21}A_1(t_n,\delta t)) \cdot & \\
\Cay(\alpha_{11}A_1(t_n,\delta t) - \alpha_{12}A_2(t_n,\delta t))Y_n&
\end{aligned}
\]
for a given time stepping $t_n = t_0,\ldots,t_N$.
\end{remark}

\begin{remark}
In contrast to the exponential commutator-free method \cite{AF:2011}, we cannot expect a fourth-order scheme with the product of only two Cayley transforms. Indeed, if we write 
$Y_1 = \Cay(\Om_1(\delta t)) \Cay(\Om_2(\delta t))Y_0 = \Cay(\widetilde \Om(\delta t))Y_0$, then the coefficients $\alpha_{i,k}$ will be complex. Complex coefficients, however, are not compatible with the structure of the Lie algebra $\mathfrak g$, which is a vector space over $\mathbb R$.
\end{remark}

%---------------------------------------
%---------------------------------------
\section{Examples}
\label{sec:Examples}
%---------------------------------------
%---------------------------------------

In this section we consider two examples to illustrate our results. The first one, a driven two-level quantum system, is a classical and well known system in quantum computing. It appears as a good starting point since it has also been considered in the context of the exponential commutator-free approach \cite{AF:2011}. So, we will be able to compare our scheme with the fourth-order commutator-free exponential time-propagator CFET4:2 derived in \cite{AF:2011} for a system where we have an analytical solution. 
In the second example we consider the Schrödinger equation with a time-dependent Hamiltonian in dimension one. 

Time-dependent Schrödinger equations with explicitly time-dependent potentials occur naturally in quantum optimal control, where the potential generally contains a time-dependent control function (laser profile, magnetic field, etc.). Several approaches based on splitting methods have been developed for this type of problems (see for instance \cite{BCM2008, IKP, Pranav2019}). The interest of our approach is to propose an alternative to the use of exponentials to integrate the solution. 
Indeed, in the context of quantum optimal control, the search for a solution involves an optimization process that usually requires a significant number of system integrations. The use of Cayley transforms instead of exponentials therefore makes sense, given that the geometric properties of the system are preserved, while the solution is computed more cheaply. 

For the first example, we consider a magnetic field such that we can analytically express the solution, which is taken as the reference solution. In the second example, we will consider CFET4:2 from \cite{AF:2011} as the reference solution.

%---------------------------------------
\subsection{A driven two-level system}

For our first test problem, we consider an example from \cite{AF:2011} which is a driven two-level system, realized for instance by a spin $1/2$ in a magnetic field $\vec B(t) = (B_x(t), B_y(t), B_z(t))$. In the eigenbasis of the z-component of angular momentum, the Hamiltonian operator is defined by
\[
H(t) = \frac{1}{2} \begin{pmatrix}
B_z(t) & B_x(t) - i B_y(t) \\
B_x(t) + i B_y(t) & -B_z(t)
\end{pmatrix}.
\]
With the magnetic field $\vec B(t) = (2V \cos(2\omega t), 2V \sin(2\omega t), 2 \Delta)$ the system is periodically driven and the propagator can be analytically expressed using Floquet theory. In this particular case, the Hamiltonian becomes
\[
H(t) = \frac{1}{2} \begin{pmatrix}
\Delta & V e^{-2i\omega t} \\
V e^{2i\omega t} & -\Delta
\end{pmatrix},
\]
where $\Delta, V, \omega \in \R$. The system to solve is given by
\begin{equation}
\label{eq:syst-1}
\dot{Y}(t) = A(t)Y(t), \quad Y(t_0) = I, \quad  \quad A(t) = -i H(t),
\end{equation}
with $Y(t) \in \mathrm{U}_2(\C) \subset \C^{2\times 2}$ and $t \in [0, T]$. The exact solution is given by
\[
Y(t) = \begin{pmatrix}
e^{-i w t}\left(\cos(\Lambda t) - i\frac{\Delta-\omega}{\Lambda}\sin(\Lambda t) \right) & -i\frac{V}{\Lambda}e^{-i w t}\sin(\Lambda t)  \\
& \\
-i\frac{V}{\Lambda}e^{i w t}\sin(\Lambda t) & e^{i w t}\left(\cos(\Lambda t) + i\frac{\Delta-\omega}{\Lambda}\sin(\Lambda t) \right)
\end{pmatrix},
\]
with $\Lambda = \sqrt{(\Delta - \omega)^2 + V^2}$. 
Notice that in accordance with Floquet theory for periodically driven systems, $Y(\pi n/\omega, 0) = Y(\pi/\omega, 0)^n$ for any given $n \in \mathbb{N}$. Also, the transition probability spin up $\leftrightarrow$ spin down,
\[
P(t) = |Y_{21}(t,0)|^2 = \left(\frac{V}{\Om}\right)^2 \sin^2(\Om t),
\]
is typical for a Breit--Wigner resonance.

For the numerical simulations, we solve the system \eqref{eq:syst-1} using different numerical schemes. The first one is the Commutator-Free Exponential Time-propagator (denoted as CFET4:2) from \cite[Prop.~4.2]{AF:2011} which is the exponential version of the main method developed in this work. We use also the Cayley--Magnus Time-propagator (denoted as CMT), given by equation \eqref{eq:simplify-approx} and already obtained by Iserles in \cite{Iserles:2001}, which contains nested commutators. Finally we use the main method derived is this paper namely the Commutator-Free Cayley Time-propagator (denoted as CFCT) and given by 
\[
\begin{aligned}
Y(t_{n+1}) \approx Y^{\mathrm{[CFCT]}}_{n+1} = \Cay(\alpha_{11}A_1(t_n,\delta t) + \alpha_{12}A_2(t_n,\delta t)) \Cay(\alpha_{21}A_1(t_n,\delta t)) \cdot & \\
\Cay(\alpha_{11}A_1(t_n,\delta t) - \alpha_{12}A_2(t_n,\delta t)) Y_n. &
\end{aligned}
\]
We propagating until $T = 20\pi/\omega$, taking $\omega = 1$, $\Delta = V = 0.5$. For the error analysis we consider the Euclidean norm in $\C^2$. The solutions as well as error and total energy during the propagation are displayed in Figure~\ref{fig:driven-two-level-system-solution}. Conservation of the norm (also display in the same figure), insure the preservation of the transition probability $P(t)$ by the numerical scheme, which is not the case when using a classical scheme as the Runge--Kutta method RK45.

\begin{figure}[ht!]  
\begin{minipage}{0.47\textwidth}
\def\sizefig{1.}
\includegraphics[width=\sizefig\textwidth]{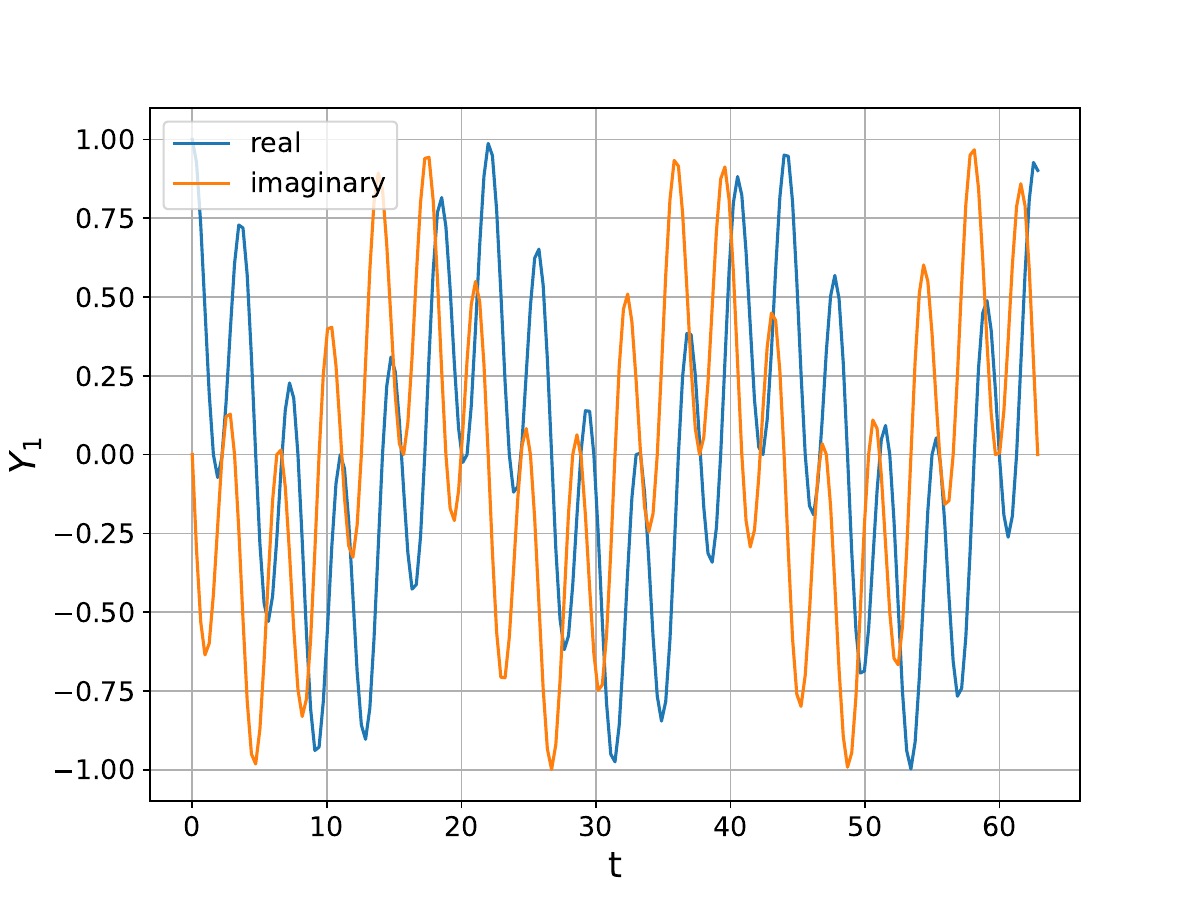}
\includegraphics[width=\sizefig\textwidth]{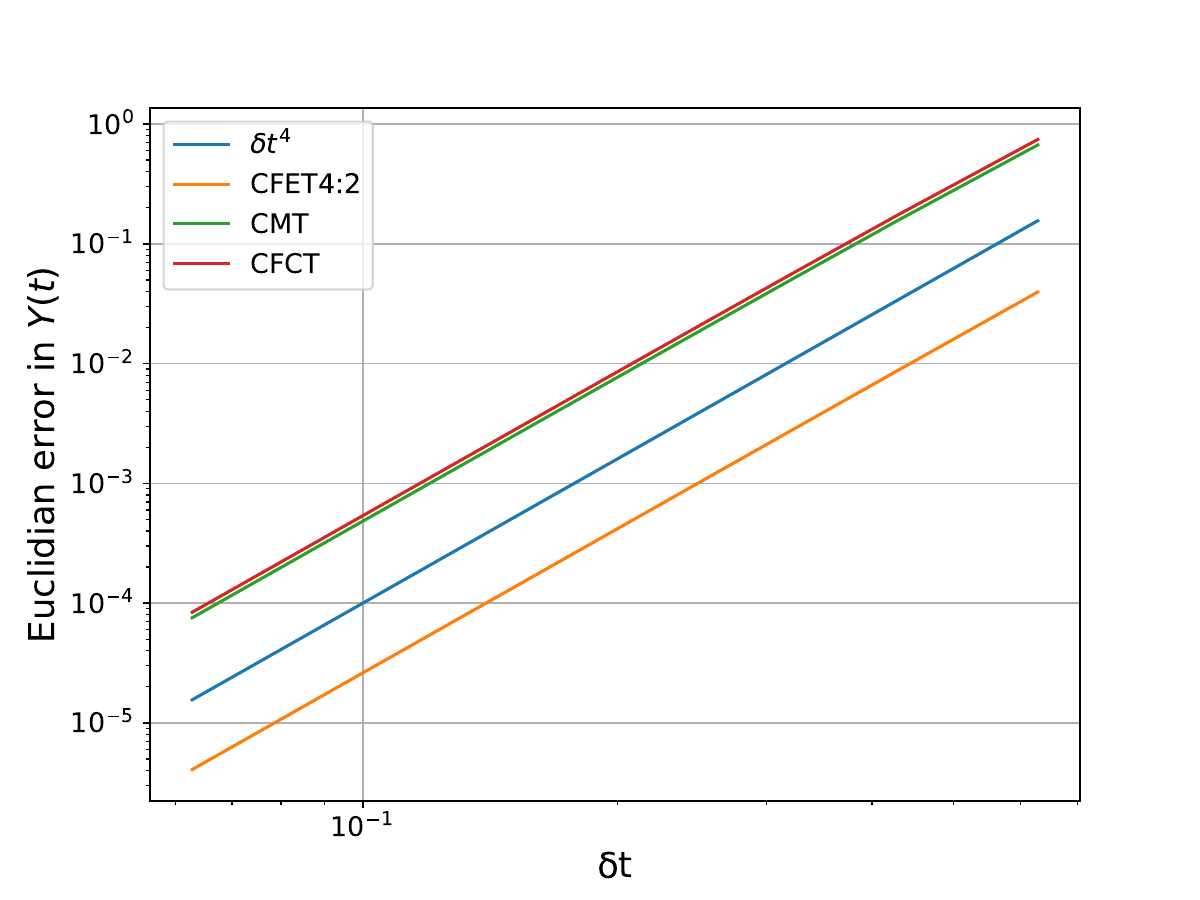}
\end{minipage}
\hspace{0.2cm}
%\hfill
\begin{minipage}{0.47\textwidth}
\def\sizefig{1.}
\includegraphics[width=\sizefig\textwidth]{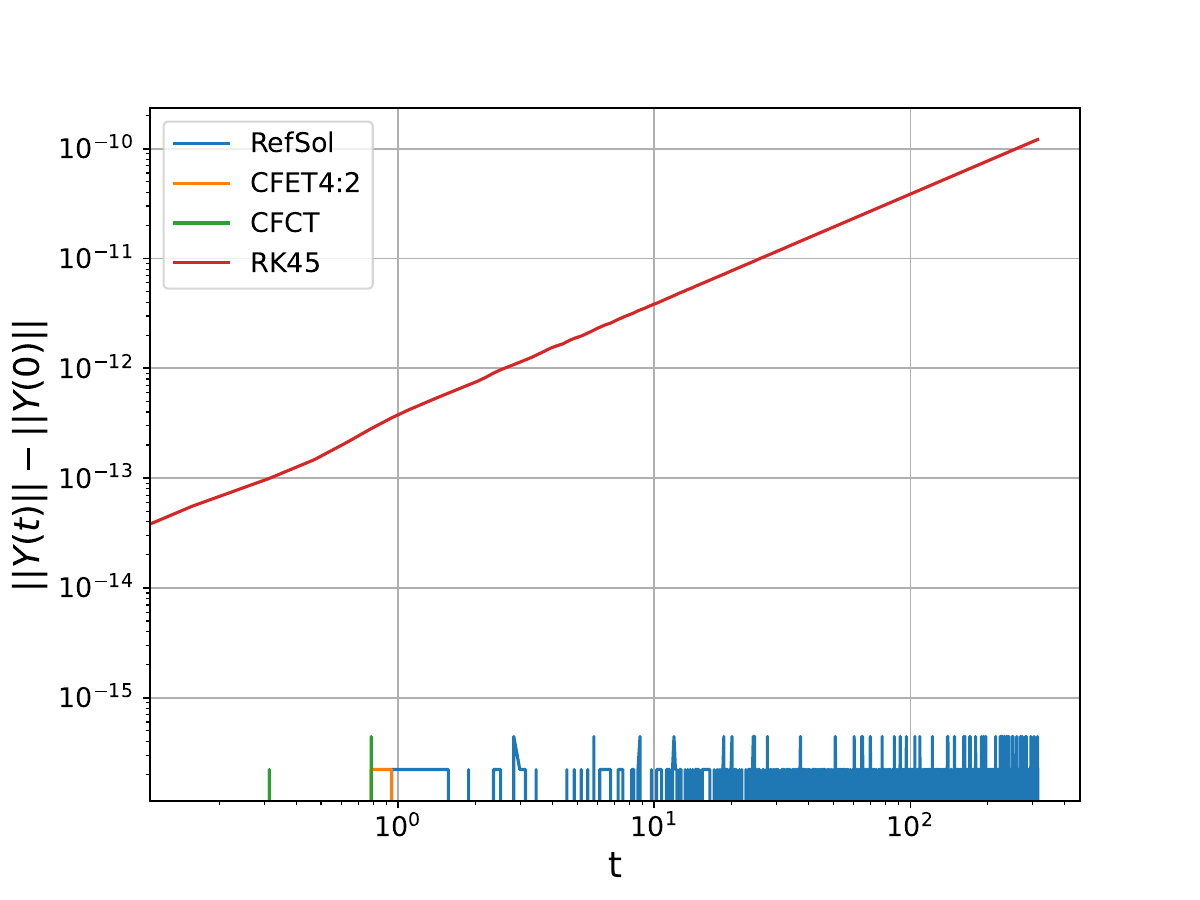}
\includegraphics[width=\sizefig\textwidth]{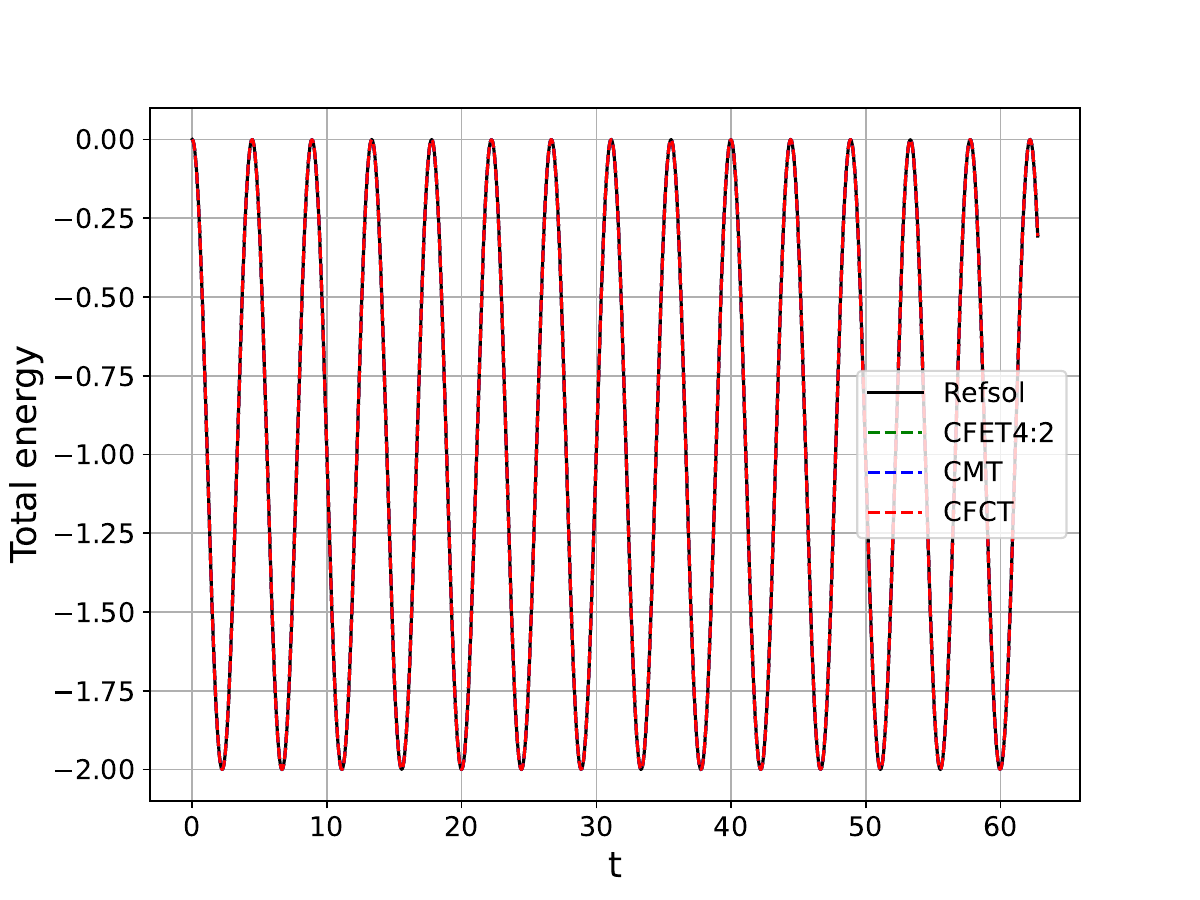}
\end{minipage}
\caption{(Left) Projection of the solution of system \eqref{eq:syst-1} along the first axis, together with error obtained for CFET4:2, CMT, CFCT taking $T = 20\pi/\omega$, $\omega = 1$ and $\Delta = V = 0.5$. (Right) Illustration of the norm conservation during the propagation for CFET4:2, CMT and CFCT. We can clearly see the loss of this property by using the classical integrator as RK45 (taking $rtol = atol = 1e^{-12}$)}.
\label{fig:driven-two-level-system-solution}
\end{figure}

%---------------------------------------
\subsection{Linear time-dependent Schrödinger equation}
\label{Ex:2}

We consider the one-dimensional time-dependent Schrödinger equation,
\begin{equation}
\label{eq:SE}
i\partial_t \varphi(x,t) = H(x,t)\varphi(x,t), \quad \quad \varphi(x,0) = \varphi_0(x) \quad \quad t \in [0,T], \quad x \in D = [-L,L],
\end{equation}
where the time-dependent Hamiltonian $H(x,t)$ is in the form
\[
H(x,t) = \partial_x^2 + V(x, t), 
\]
with the potential $V(x,t) = V_0(x) + u(t)x$ containing an external potential (with a fixed control term $u(t)$ in this example). 
In this first case, we consider the internal potential $V_0$ and the laser profile $u(t)$ to be given by
\[
V_0(x) = x^4 - 10 x^2, \quad u(t) = c \sin(\omega t), ~~ c =-10^2, ~~ \omega = 5\pi. 
\]
We consider as initial state a Gaussian wave-packet, 
\[
\varphi_0(x) = \mathrm{e}^{-\frac{(x-\bar x_0})^2}{2\sigma^2}, \quad \sigma = 0.5, \quad {\bar x_0} = -2.
\]
Following spatial discretisation, we have to solve the following system of ODEs,
\[
\partial_t \varphi(t) = A(t)\varphi(t), \quad \varphi(0) = \varphi_0, \quad A(t) = -i H(t),
\]
where $H(t)$ is now a matrix representation of the Hamiltonian. Specifically, we use a Fourier spectral discretization on an equispaced grid $-L = x_0, \cdots, x_N = L$, after imposing periodic boundary conditions. {For implementation we fix $L=10$, $N = 250$ for space discretization and $M=200$ for time discretization.} 

We implement both CMT (with commutators) and CFCT (with commutator-free) and compare both with the reference solution (which is given here by CFET4:2). 
For a time discretization $0 = t_0, \cdots, t_M = T$, CMT and CFCT respectively read
\[
\varphi^{\mathrm{[CMT]}}_{n+1} = \Cay(A_1(t_n, \delta t) - \frac{1}{6}[A_1(t_n, \delta t), A_2(t_n, \delta t)] - \frac{1}{12} A_1^3(t_n, \delta t)\varphi^{\mathrm{[CMT]}}_n,
\]
\[
\varphi^{\mathrm{[CFCT]}}_{n+1} = \Cay \left(\alpha_{11}A_1 + \alpha_{12}A_2 \right) \, \Cay \left(\alpha_{21}A_1\right) \, \Cay \left(\alpha_{11}A_1 - \alpha_{12}A_2\right) \varphi^{\mathrm{[CFCT]}}_{n},
\]
with $\alpha_{11}, \alpha_{12}, \alpha_{21}, \alpha_{22}, \alpha_{31}, \alpha_{32}$ defined in \eqref{eq:coefs}, where
\[
A_1 = \frac{\delta t}{2}(A^1 + A^2), \quad  A_2 = \frac{\delta t\sqrt{3}}{2}(A^2 - A^1),
\]
and
\[
A^1 = A\left(t_n + \left(\frac{1}{2}-\frac{\sqrt{3}}{6} \right) \delta t\right), \quad A^2 = A\left(t_n + \left(\frac{1}{2}+\frac{\sqrt{3}}{6} \right) \delta t \right).
\]

\paragraph{Numerical simulations.} We check the conservation of the state norm $\norm{\varphi(\cdot,t)}_{L^2(D)}$ for each of these two methods and compute the error with respect to CFET4:2  given in Figure~\ref{fig:SE}. 
In addition, we compute the energy change during propagation, Figure~\ref{fig:SE}. Note that, since we have a time-dependent Hamiltonian, energy is no longer conserved during propagation. In Figure~\ref{fig:SE-energy} one can see the blow up of the energy when considering RK45 scheme.

The implementation of these methods is done using the \textsc{expsolve} package \cite{expsolve}, which is utilized for initializing the Hamiltonian, computing observables such as the energy, and the computation of the $L^2$-norms and inner products (using the \texttt{l2norm} and \texttt{l2inner} methods).

\begin{remark}
{
Although the CFCT method appears less accurate than the CMT approach in this second example, its computational time is more than six times faster. On the other hand, both methods exhibit a lower accuracy than the expected reference slope for $\delta t^4$. This discrepancy can be partly attributed to the choice of the reference solution.
}
\end{remark}

\begin{figure}[ht!]  
\def\sizefig{1.0}
\begin{minipage}{0.475\textwidth}
\includegraphics[width=\sizefig\textwidth]{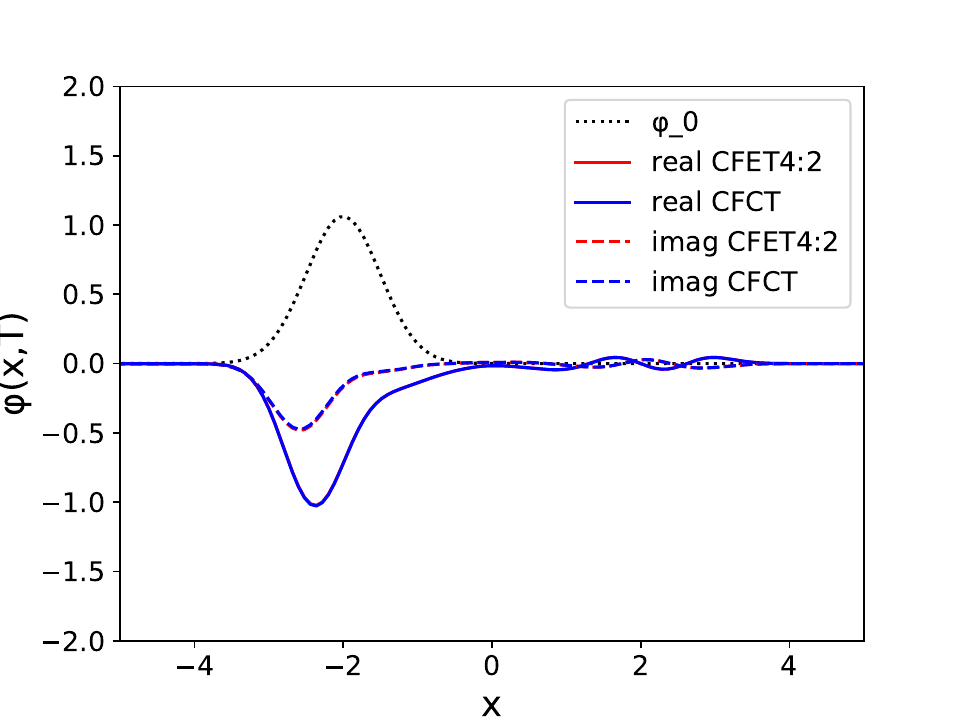}
\includegraphics[width=\sizefig\textwidth]{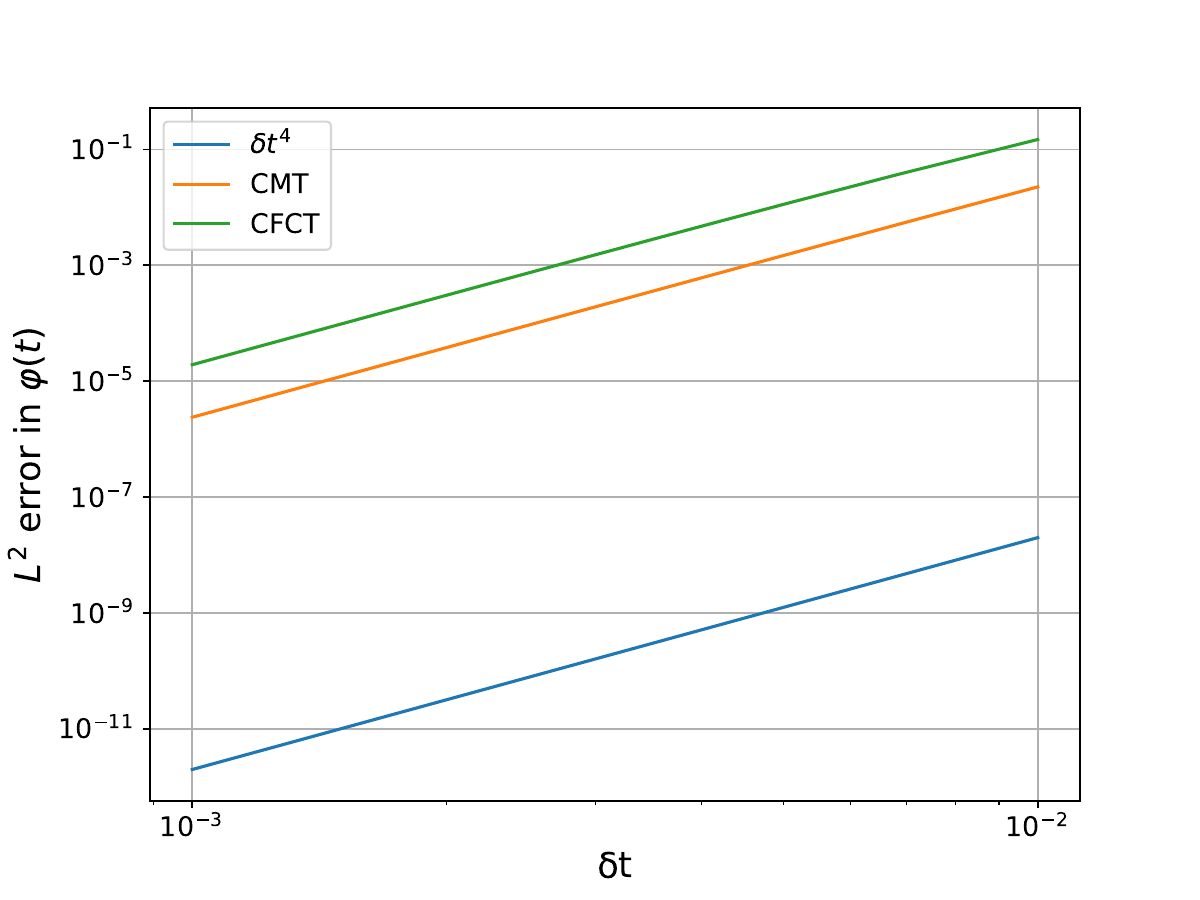}
\end{minipage}
\hspace{0.2cm}
%\hfill
\begin{minipage}{0.475\textwidth}
\includegraphics[width=\sizefig\textwidth]{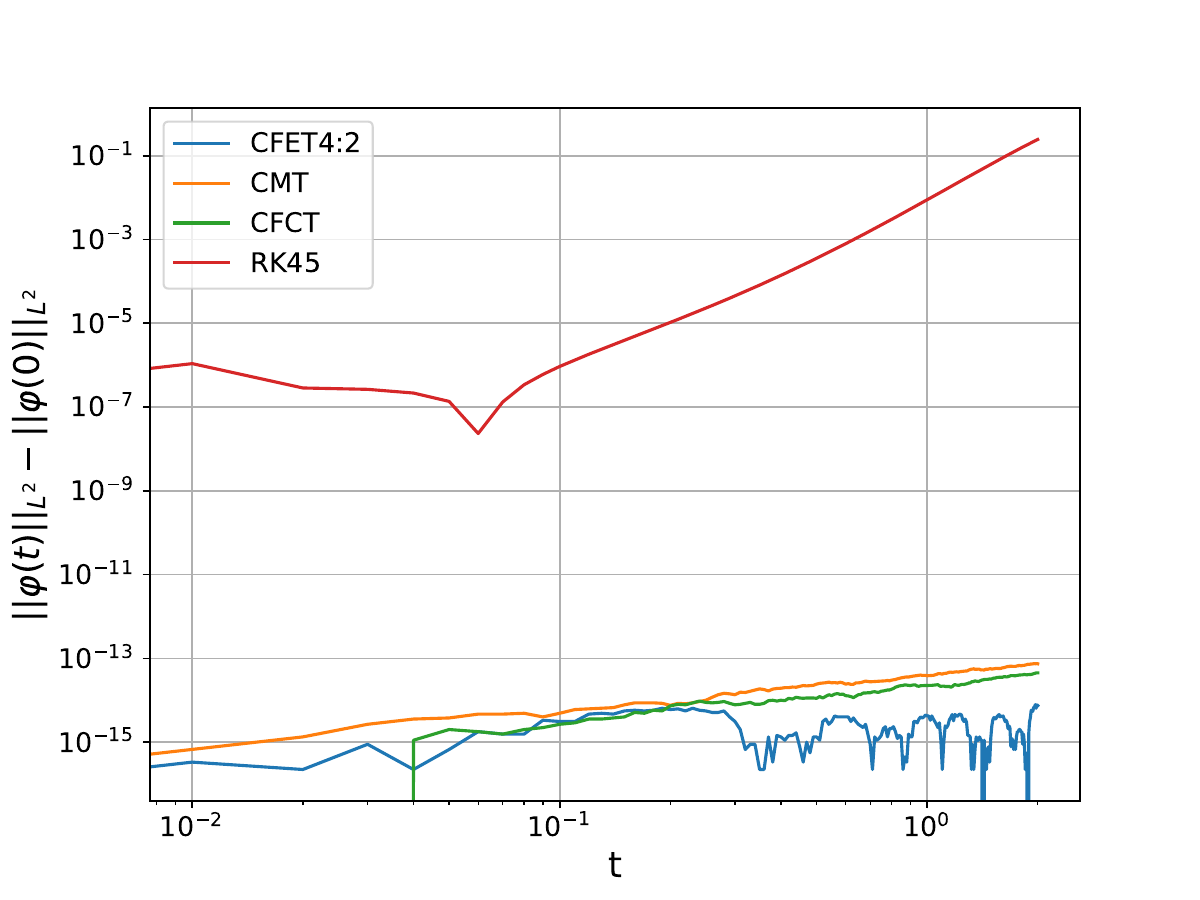}
\includegraphics[width=\sizefig\textwidth]{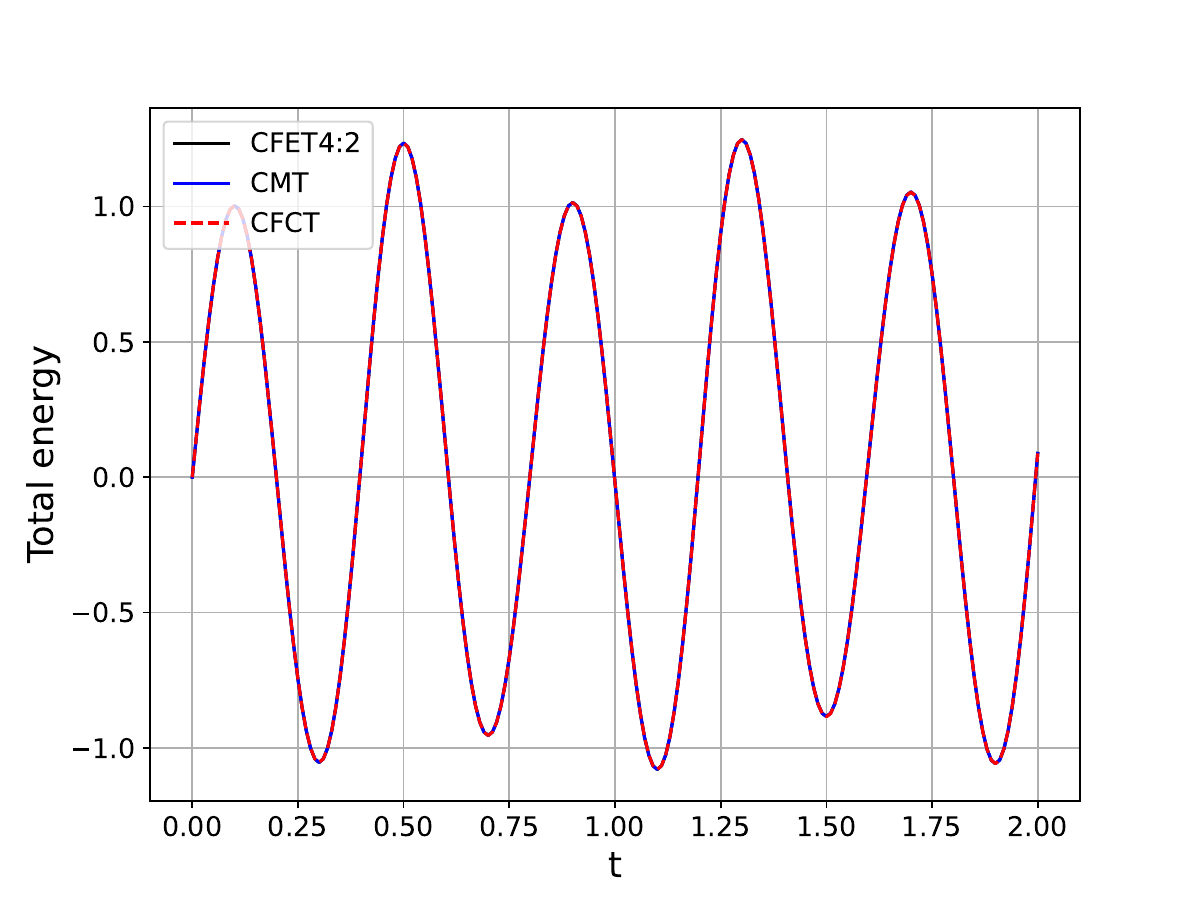}
\end{minipage}
\caption{(Left) Solution of system \eqref{eq:SE} together with error obtained for CMT, CFCT (taking CFET4:2 as reference solution), propagated until $T = 2$. (Right) Illustration of the norm conservation and change of the energy during the propagation for CFET4:2, CMT and CFCT. Again, there is no conservation of the norm along the propagation when using the classical integrator RK45.}
\label{fig:SE}
\end{figure}

\begin{figure}[ht!]  
\centering
\includegraphics[width=0.6\textwidth]{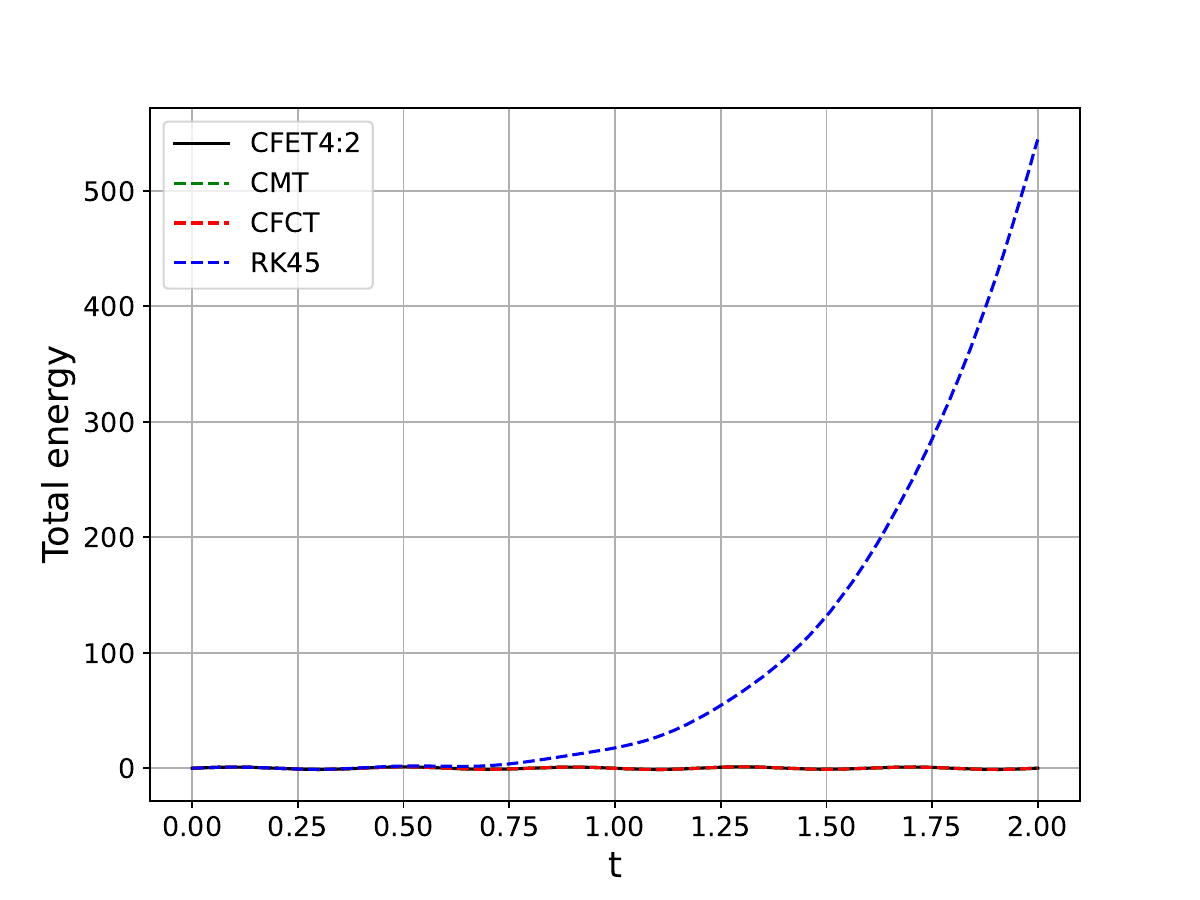}
\caption{Illustration of the energy blow-up when using RK45 to solve example 2.}
\label{fig:SE-energy}
\end{figure}

%---------------------------------------
%---------------------------------------
%---------------------------------------
\section{Conclusion}
In this article, we introduced a new numerical method for solving differential systems evolving on quadratic Lie groups, with particular applications to quantum systems. Our approach combines the Cayley-Magnus expansion with commutator-free techniques. A key intermediate result in the derivation of our method is the formulation of a Cayley-version of the Baker–Campbell–Hausdorff (BCH) formula, which we explicitly computed for up to three matrices.
One of the main advantages of this method lies in the significant computational speed-up it offers, owing to the use of the Cayley transform in place of the matrix exponential. This feature is especially beneficial in quantum optimal control, where optimization algorithms often involve repeated integrations of quantum systems.
In this work, we proposed a method of order four. Looking ahead, we aim to extend our approach to construct high-order methods and to implement these techniques within numerical algorithms tailored for quantum optimal control problems.

%---------------------------------------
\section*{Acknowledgement}

CO acknowledges financial support by the Ministry of Culture and Science of the State of North
Rhine-Westphalia, Germany. SM acknowledges financial support by Deutsche Forschungsgemeinschaft (DFG), Grant No. OB 368/5-1, AOBJ: 692093.
We would also like to thank Arieh Iserles for insightful discussions on his visit to Paderborn University in November 2024.

\appendix
\section{Details on the computation of coefficients} 
\label{app1}

Writing down all the equations in \eqref{eq:BCH-develop} and equalizing them with relation \eqref{eq:simplify-approx} one arrives at the system
\[
\left\{\begin{aligned}
& \alpha_{11} + \alpha_{21} + \alpha_{31} = 1 \\
& \alpha_{12} + \alpha_{22} + \alpha_{32} = 0 \\
& \alpha_{11}\alpha_{22} + \alpha_{11}\alpha_{32} + \alpha_{21}\alpha_{32} - \alpha_{12}\alpha_{21} - \alpha_{12}\alpha_{31} - \alpha_{22}\alpha_{31} = - \frac{1}{3} \\
& 2\alpha_{11}\alpha_{21}\alpha_{31} + \alpha_{11}^2\alpha_{21} + \alpha_{11}\alpha_{21}^2 + \alpha_{11}^2\alpha_{31} + \alpha_{11}\alpha_{31}^2 + \alpha_{21}^2\alpha_{31} + \alpha_{21}\alpha_{31}^2 = \frac{1}{3} \\
& 2\alpha_{11}\alpha_{22}\alpha_{31} + \alpha_{11}\alpha_{12}\alpha_{21} + \alpha_{11}\alpha_{12}\alpha_{31} + \alpha_{11}\alpha_{21}\alpha_{22} + \alpha_{11}\alpha_{31}\alpha_{32} \\ 
& \hspace{7.cm} + \alpha_{21}\alpha_{22}\alpha_{31} + \alpha_{21}\alpha_{31}\alpha_{32}  = 0 \\
& 2\alpha_{11}\alpha_{21}\alpha_{32} + \alpha_{11}^2\alpha_{22} + \alpha_{11}^2\alpha_{32} + \alpha_{12}\alpha_{21}^2 + \alpha_{21}^2\alpha_{32} + \alpha_{11}\alpha_{31}^2 + \alpha_{22}\alpha_{31}^2 \\ 
& \hspace{6.7cm} + 2\alpha_{12}\alpha_{21}\alpha_{31} - 2\alpha_{11}\alpha_{22}\alpha_{31} = 0, \\
\end{aligned}
\right.
\]
We solved this system using  "Maple solver" for real-valued coefficients $\alpha_{11}$, $\alpha_{12}$, $\alpha_{21},$ $\alpha_{22},$ $\alpha_{31},$ $\alpha_{32},$ satisfying the required conditions given in Section~\ref{sec:Cayley-commutator-free}.

%-------------------------------------------
\section{Example \ref{Ex:2} with a different quadrature rule}
\label{App2}
To validate our assertion in Remark~\ref{remark:quadrature}, we repeat the example \ref{Ex:2} considering now the following quadrature: 
\begin{equation}
\label{eq:quadrature-rule}
\int_a^{a+\delta t} f(t) \diff t \approx W(f) = \frac{1}{3} \left(2 f\left(a + \frac{\delta t}{4}\right) - f\left(a + \frac{\delta t}{2}\right) + 2 f\left(a + \frac{3\delta t}{4}\right) \right).
\end{equation}
In Fig. B.4 it can be observed that the order of the scheme is preserved 

\begin{figure}[ht!]
\centering
\def\sizefig{0.39}
\includegraphics[width=\sizefig\textwidth]{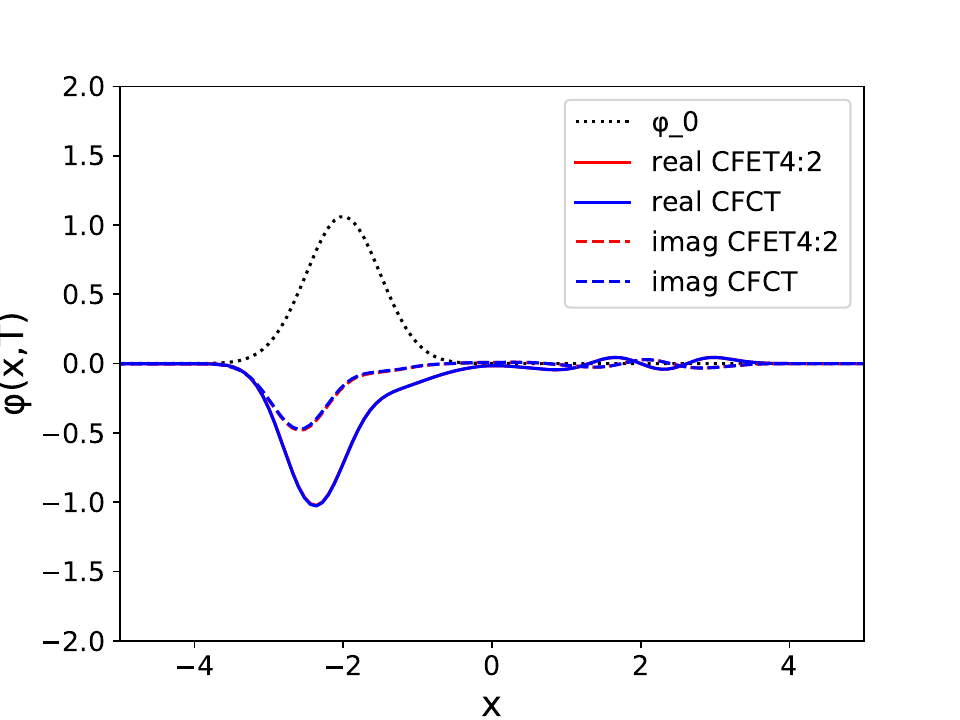}
\hspace{0.5cm}
\includegraphics[width=\sizefig\textwidth]{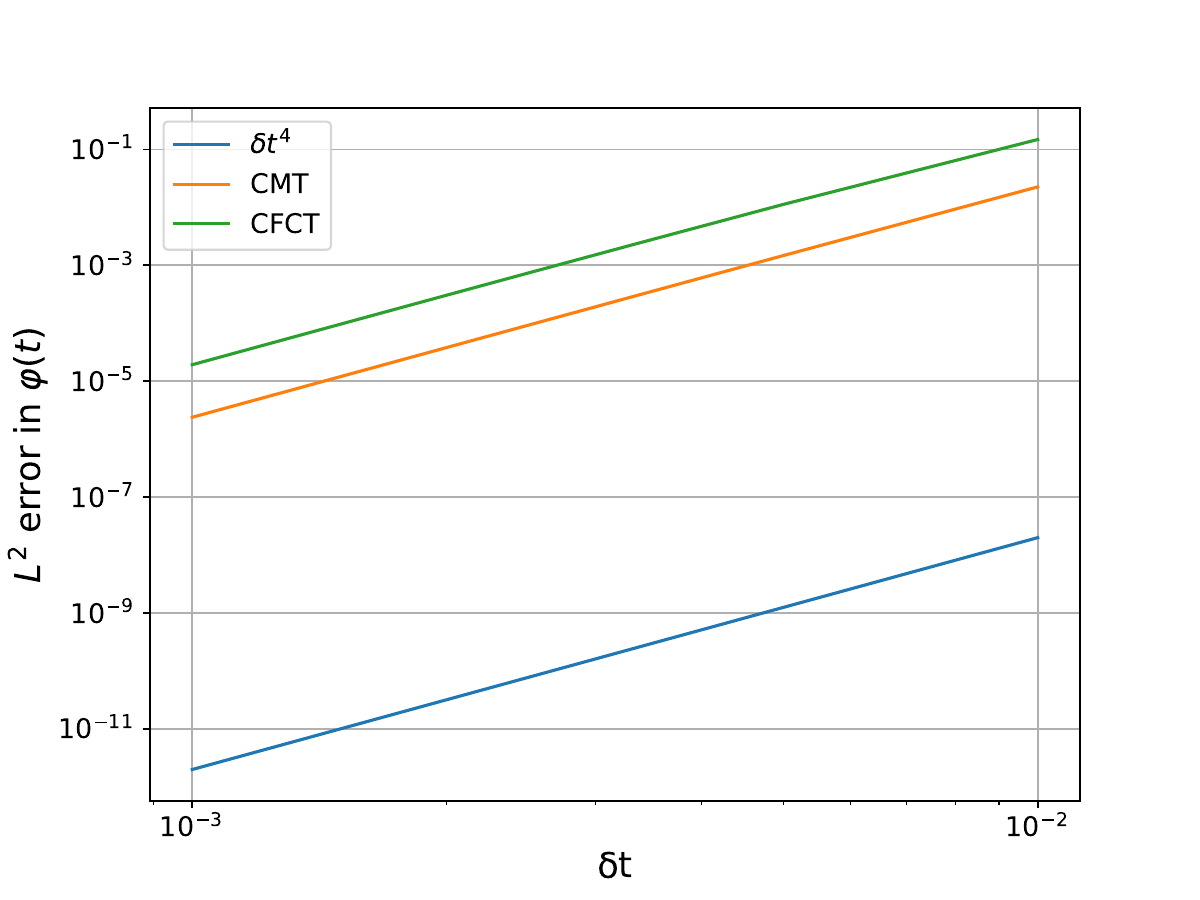}
\vspace{-0.3cm}
\caption{Solution of system (26) together with the error obtained for CMT, CFCT (taking
CFET4:2 as reference solution), with T = 2, using the quadrature rule defined in eq. (B.1).}
\label{fig:SE-new}
\end{figure}

%-------------------------------------------
\section{Details on the proof of Proposition \ref{prop:Cayley-BCH}}
\label{App3}

Recall that one has
\[
\begin{aligned}
\Om(A,B,C) = -2 \left(I - \left(I - \frac{A}{2} \right)^{-1} \left(I + \frac{A}{2} \right)\left(I - \frac{B}{2} \right)^{-1} \left(I + \frac{B}{2} \right) \cdot \right. & \\
\hspace{1.0cm} \left. \left(I - \frac{C}{2} \right)^{-1} \left(I + \frac{C}{2} \right) \right)  \cdot \left(I + \left(I - \frac{A}{2} \right)^{-1} \left(I + \frac{A}{2} \right) \cdot \right. & \\
\left. \left(I - \frac{B}{2} \right)^{-1} \left(I + \frac{B}{2} \right) \left(I - \frac{C}{2} \right)^{-1} \left(I + \frac{C}{2} \right) \right)^{-1}. &
\end{aligned}
\]
Taylor expansion of this relation gives
\[
\begin{aligned}
\Om &=  -2 \left(I - \left(I + \frac{A}{2} + \frac{A^2}{4} + \frac{A^3}{8} \right) \left(I + \frac{A}{2} \right) \left(\frac{B}{2} + \frac{B^2}{4} + \frac{B^3}{8} \right) \left(I + \frac{B}{2} \right) \right. \\
&\quad  \left. \left(\frac{C}{2} + \frac{C^2}{4} + \frac{C^3}{8} \right) \left(I + \frac{C}{2} \right) \right) \left(I + \left(I + \frac{A}{2} + \frac{A^2}{4} + \frac{A^3}{8} \right) \left(I + \frac{A}{2} \right) \right. \\
&\quad  \left. \left(\frac{B}{2} + \frac{B^2}{4} + \frac{B^3}{8} \right) \left(I + \frac{B}{2} \right) \left(\frac{C}{2} + \frac{C^2}{4} + \frac{C^3}{8} \right) \left(I + \frac{C}{2} \right) \right)^{-1} + F(A,B,C) \\
&= -2 \left(I - \left(I + A + \frac{A^2}{2} + \frac{A^3}{4} \right) \left(I + B + \frac{B^2}{2} + \frac{B^3}{4} \right) \left(I + C + \frac{C^2}{2} + \frac{C^3}{4} \right) \right) \times \\
&  \left(I + \left(I + A + \frac{A^2}{2} + \frac{A^3}{4} \right) \left(I + B + \frac{B^2}{2} + \frac{B^3}{4} \right) \left(I + C + \frac{C^2}{2} + \frac{C^3}{4} \right) \right)^{-1} + F(A,B,C) \\
&= \left(A + B + C + AB + AC + BC + \frac{A^2}{2} + \frac{B^2}{2} + \frac{C^2}{2} + \frac{A^3}{4} + \frac{B^3}{4} + \frac{C^3}{4} + \frac{A^2B}{2} + \frac{AB^2}{2} \right. \\
& \quad \left. + \frac{A^2C}{2} + \frac{AC^2}{2} + \frac{B^2C}{2} + \frac{BC^2}{2} \right) \left(I + \left(\frac{A}{2} + \frac{B}{2} + \frac{C}{2} + \frac{AB}{2} + \frac{AC}{2} + \frac{BC}{2} + \frac{A^2}{4} + \frac{B^2}{4}  \right.\right. \\
& \left.\left. + \frac{C^2}{4} + \frac{A^2B}{4} + \frac{AB^2}{4} + \frac{A^2C}{4} + \frac{AC^2}{4} + \frac{B^2C}{4} + \frac{BC^2}{4} + \frac{A^3}{8} + \frac{B^3}{8} + \frac{C^3}{8} \right)\right)^{-1} \hspace{-0.4cm} + F(A,B,C) \\
&= \left(A + B + C + AB + AC + BC + \frac{A^2}{2} + \frac{B^2}{2} + \frac{C^2}{2} + \frac{A^3}{4} + \frac{B^3}{4} + \frac{C^3}{4} + \frac{A^2B}{2} + \frac{AB^2}{2} \right. \\
& \quad \left. + \frac{A^2C}{2} + \frac{AC^2}{2} + \frac{B^2C}{2} + \frac{BC^2}{2} \right) \left( I - \frac{A}{2} - \frac{B}{2} - \frac{C}{2} - \frac{AB}{2} - \frac{AC}{2} - \frac{BC}{2} - \frac{A^2}{4} - \frac{B^2}{4} \right. \\
&\quad \left. - \frac{C^2}{4} + \left(\frac{A}{2} + \frac{B}{2} + \frac{C}{2} + \frac{AB}{2} + \frac{AC}{2} + \frac{BC}{2} + \frac{A^2}{4} + \frac{B^2}{4} + \frac{C^2}{4} \right)^2 \right) + F(A,B,C) 
\end{aligned}
\]
\[
\begin{aligned}
&= \left(A + B + C + AB + AC + BC + \frac{A^2}{2} + \frac{B^2}{2} + \frac{C^2}{2} + \frac{A^3}{4} + \frac{B^3}{4} + \frac{C^3}{4} + \frac{A^2B}{2} + \frac{AB^2}{2} \right. \\
& \quad \left. + \frac{A^2C}{2} + \frac{AC^2}{2} + \frac{B^2C}{2} + \frac{BC^2}{2} \right) \left( I - \frac{A}{2} - \frac{B}{2} - \frac{C}{2} - \frac{AB}{4} - \frac{AC}{4} - \frac{BC}{4} + \frac{BA}{4} \right. \\
&\quad \left. + \frac{CA}{4} + \frac{CB}{4} \right) + F(A,B,C) \\
&= A + B + C + \frac{1}{2}\left([A,B] + [A,C] + [B,C] \right) + \frac{1}{4}[A,B],C] - \frac{1}{4}(ACB + BCA) \\
&\quad\quad  - \frac{1}{4}(ABA + BAB + BAB + BCB + CAC + CBC) + F(A,B,C),
\end{aligned}
\]   
with the terms in $F(A,B,C)$ being of at least fourth-order. Since we assumed that $A,B$ and $C$ are in the neighborhood of zero, the series converges.} Moreover, the Neumann series for $(I-A/2)^{-1}$ converges so that the term \\ $\left( I + \Cay(A) \Cay(B) \Cay(C)\right)^{-1}$ becomes:
\[
\begin{aligned}
& \left(2I + \left(A + B + C + AB + AC + BC + \frac{A^2}{2} + \frac{B^2}{2} + \frac{C^2}{2} \right.\right. \\
& \left.\left.  + \frac{A^2B}{2} + \frac{AB^2}{2} + \frac{A^2C}{2} + \frac{AC^2}{2} + \frac{B^2C}{2} + \frac{BC^2}{2} + \frac{A^3}{4} + \frac{B^3}{4} + \frac{C^3}{4} \right) + \widetilde F(A,B,C) \right)^{-1}  \\
=& \frac{1}{2}\left(I + \frac{1}{2}\left(A + B + C + AB + AC + BC + \frac{A^2}{2} + \frac{B^2}{2} + \frac{C^2}{2} \right.\right. \\
& \left.\left.  + \frac{A^2B}{2} + \frac{AB^2}{2} + \frac{A^2C}{2} + \frac{AC^2}{2} + \frac{B^2C}{2} + \frac{BC^2}{2} + \frac{A^3}{4} + \frac{B^3}{4} + \frac{C^3}{4} \right) + \widetilde F(A,B,C) \right)^{-1},    
\end{aligned}
\]
where $\widetilde F(A,B,C)$ contains terms related to truncated terms of the series expansion of $(I- \frac A2)^{-1}$, $(I- \frac B2)^{-1}$, and $(I- \frac C2)^{-1}$. Once again, these converge and $\widetilde F(A,B,C)$ is small. Thus, $\frac 12 (I + \Cay(A) \Cay(B) \Cay(C))$ is a small perturbation of the identity such that the Neumann series to $\left( I + \Cay(A) \Cay(B) \Cay(C)\right)^{-1}$ converges.

%-------------------------------
%-------------------------------
%-------------------------------
\bibliographystyle{siamplain}

\begin{thebibliography}{00}

\bibitem[Alvermann11]{AF:2011}
A.~Alvermann, H.~Fehske,
\newblock High-order commutator-free exponential time-propagation of driven quantum systems,
\newblock J. Comput. Phys., \textbf{230} (2011),
pp.~5930-5956,
\href{https://api.semanticscholar.org/CorpusID:35139068}{DOI:10.1016/j.jcp.2011.04.006}.


\bibitem[Auzinger16]{AKKT}
W.~Auzinger, T.~Kassebacher, O.~Koch and M.~Thalhammer,
\newblock Adaptive splitting methods for nonlinear Schrödinger equations in the semiclassical regime,
\newblock Numer Algor \textbf{72} (2016), pp.~1–35, \href{https://doi.org/10.1007/s11075-015-0032-4}{DOI:10.1007/s11075-015-0032-4}
% this is to motivate in the introduction


\bibitem[Auzinger22]{ADHHJ2022}
W.~Auzinger, J.~Dubois, K.~Held, H.~Hofstätter, T.~Jawecki, A.~Kauch, O.~Koch, K.~Kropielnicka, P.~Singh, C.~Watzenböck,
\newblock Efficient Magnus-type integrators for solar energy conversion in Hubbard models,
\newblock Journal of Computational Mathematics and Data Science, \textbf{2},
(2022), 100018, \href{https://doi.org/10.1016/j.jcmds.2021.100018}{DOI:10.1016/j.jcmds.2021.100018}.

\bibitem[Baum83]{BTP1983}
J.~Baum, R.~Tycko, A.~Pines,
\newblock Broadband population inversion by phase modulated
pulses, 
\newblock J. Chem. Phys. \textbf{79} (1983), no.~9, pp.~4643–4644,
\href{https://doi.org/10.1063/1.446381}{DOI:10.1063/1.446381}.


\bibitem[Blane16]{BC:2017}
S.~Blanes, and F.~Casas, 
\newblock A concise introduction to geometric numerical integration (1st ed.), 
\newblock Chapman and Hall/CRC (2016). \href{https://doi.org/10.1201/b21563}{DOI:10.1201/b21563}


\bibitem[Blane24]{BCM2008}
S.~Blanes, F.~Casas, and A.~Murua,
\newblock Splitting methods for differential equations, \newblock arXiv:2401.01722 (2024), \href{https://doi.org/10.48550/arXiv.2401.01722}{DOI:10.48550/arXiv.2401.01722}.


\bibitem[Blane09]{BCOR:2009}
S. Blanes, F. Casas, J. A. Oteo, J. Ros, 
\newblock The Magnus expansion and some of its applications, 
\newblock Physics Reports \textbf{470} (2009), no.~151, \href{https://doi.org/10.1016/j.physrep.2008.11.001}{DOI:10.1016/j.physrep.2008.11.001}.


\bibitem[Blane06]{BM2006}
S.~Blanes, P.C.~Moan,
\newblock Fourth- and sixth-order commutator-free Magnus integrators for linear and non-linear dynamical systems,
\newblock Applied Numerical Mathematics, \textbf{56} (2006), no.~12, pp.~1519-1537, \href{https://doi.org/10.1016/j.apnum.2005.11.004}{DOI:10.1016/j.apnum.2005.11.004}.


\bibitem[Brif10]{BCR2010}
C.~Brif, R.~Chakrabarti and H.~Rabitz 
\newblock Control of quantum phenomena: past, present and future
\newblock New J. Phys. \textbf{12} (2010), 075008, \href{https://doi.org/10.1088/1367-2630/12/7/075008}{DOI:10.1088/1367-2630/12/7/075008}.


\bibitem[Celledoni22]{CCLMO2022}
E.~Celledoni, E.~Çokaj, A.~Leone, D.~Murari and B.~Owren
\newblock Lie group integrators for mechanical systems, 
\newblock International Journal of Computer Mathematics, \textbf{99} (2022), no.~1, pp.~58-88, \href{https://doi.org/10.1080/00207160.2021.1966772}{DOI:10.1080/00207160.2021.1966772}.


\bibitem[Chen23]{CFBS2023}
G.~Chen, M.~Foroozandeh, C.~Budd, P.~Singh,
\newblock Quantum simulation of highly-oscillatory many-body Hamiltonians for near-term devices (2023),
\newblock 	arXiv:2312.08310
\href{https://doi.org/10.48550/arXiv.2312.08310}{DOI:10.48550/arXiv.2312.08310}.


\bibitem[Dieci94]{Dieci1994}
L.~Dieci, R.D.~Russell, E.~Van Vleck,
\newblock Unitary Integrators and Applications to Continuous Orthonormalization Techniques,
\newblock SIAM Journal on Numerical Analysis, \textbf{31} (1994), no.~1, pp.~261-281, \href{https://doi.org/10.1137/0731014}{DOI:10.1137/0731014}.


\bibitem[Diele98]{DLP:1998}
F.~Diele, L.~Lopez and R.~Peluso,
\newblock The Cayley transform in the numerical solution of unitary differential systems,
\newblock Advances in Computational Mathematics \textbf{8} (1998), pp~317–334, \href{https://doi.org/10.1023/A:1018908700358}{DOI:10.1023/A:1018908700358}.


\bibitem[Singh21]{FP2021}
M.~Foroozandeh, P.~Singh
\newblock Optimal Control of Spins by Analytical Lie Algebraic Derivatives
\newblock Automatica \textbf{129} (2021), no.~109611
\href{https://doi.org/10.1016/j.automatica.2021.109611}{DOI:10.1016/j.automatica.2021.109611}.


\bibitem[Glaser15]{GBCKKK2015}
S.J.~Glaser, U.~Boscain, T.~Calarco, C.P.~Koch, W.~Kckenberger, R.~Kosloff,
I.~Kuprov, B.~Luy, S.~Schirmer, T.~Schulte-Herbrggen, D.~Sugny, F.K.~Wilhelm,
\newblock Training schrdingers cat: quantum optimal control, 
\newblock Eur. Phys. J. D \textbf{69} (2015), NO.~12.
\href{https://doi.org/10.1140/epjd/e2015-60464-1}{DOI:10.1140/epjd/e2015-60464-1}.


\bibitem[Güttel13]{Guttel2013}
S. Güttel, 
\newblock Rational Krylov approximation of matrix functions: Numerical methods and optimal pole selection, 
\newblock GAMM-Mitt., \textbf{36} (2013), pp.~8–31, \href{ https://doi.org/10.1002/gamm.201310002}{DOI:10.1002/gamm.201310002}


\bibitem[Grimm12]{Grimm}
V.~Grimm,
\newblock Resolvent Krylov subspace approximation to operator functions,
\newblock BIT. Numerical Mathematics \textbf{52} (2012), pp.~639-659, \href{https://doi.org/10.1007/s10543-011-0367-8}{DOI:10.1007/s10543-011-0367-8}.


\bibitem[Hall03]{Hall2003}
B.C.~Hall, 
\newblock Lie Groups, Lie Algebras, and Representations, 
\newblock Graduate Texts in Mathematics, Springer, New York (2003), \href{https://doi.org/10.1007/978-3-319-13467-3}{DOI:10.1007/978-3-319-13467-3}.


\bibitem[Hänggi97]{Hanggi:1997}
P.~Hänggi, 
\newblock Driven quantum systems, 
\newblock in: T. Dittrich, P. Hänggi, G.-L. Ingold, B. Kramer, G. Schön, W. Zwerger (Eds.), Quantum Transport and Dissipation, Wiley-VCH, Weinheim, 1997, pp. 249–286, \href{https://books.google.de/books?id=pRXLzQEACAAJ}{https://books.google.de/books}.


\bibitem[Hairer06]{HLW:2006}
E. Hairer, C. Lubich, G. Wanner, 
\newblock Geometric Numerical Integration, 
\newblock Springer, Berlin (2006), \href{https://doi.org/10.1007/3-540-30666-8}{DOI:10.1007/3-540-30666-8}



\bibitem[Hochbruck03]{HL2003}
M.~Hochbruck and C.~Lubich
\newblock On Magnus Integrators for Time-Dependent Schrödinger Equations
\newblock SIAM Journal on Numerical AnalysisVol. \textbf{41} (2003), no.~3, \href{https://doi.org/10.1137/S0036142902403875}{DOI:10.1137/S0036142902403875}.


\bibitem[Hochbruck97]{HL1997}
M.~Hochbruck and C.~Lubich,
\newblock On Krylov Subspace Approximations to the Matrix Exponential Operator,
\newblock SIAM Journal on Numerical AnalysisVol. \textbf{34} (1997), no.~5, \href{https://doi.org/10.1137/S0036142995280572}{DOI:10.1137/S0036142995280572}.


\bibitem[Hohenester14]{HRBS2007}
U.~Hohenester, P.K.~Rekdal, A.~Borzì, J.~Schmiedmayer, 
\newblock Optimal quantum control of Bose-Einstein condensates in magnetic microtraps,
\newblock Physical Review A, \textbf{75} (2014), Issue 2, id.023602, \href{https://doi.org/10.1103/PhysRevA.90.033628}{DOI:10.1103/PhysRevA.90.033628}.


\bibitem[Iserles01]{Iserles:2001}
Iserles,
\newblock On Cayley-transform methods for the discretization of Lie-Group Equations,
c\newblock Found. Comput. Math. \textbf{1} (2001) pp.~129-160, \href{https://doi.org/10.1007/s102080010003}{DOI:10.1007/s102080010003}.


\bibitem[Iserles18]{IKP}
Iserles, K.~Kropielnicka, and P.~Singh
\newblock Magnus--Lanczos Methods with Simplified Commutators for the Schrödinger Equation with a Time-Dependent Potential
\newblock SIAM Journal on Numerical AnalysisVol. \textbf{56} (2018), no.~3 \href{https://doi.org/10.1137/17M1149833}{DOI:10.1137/17M1149833}


\bibitem[Bader16]{BIKP2016}
P.~Bader, Iserles, K.~Kropielnicka, and P.~Singh
\newblock Efficient methods for linear Schrödinger equation in the semiclassical regime with time-dependent potential
\newblock Proc. R. Soc. A.47220150733 (2016),
\href{http://doi.org/10.1098/rspa.2015.0733}{DOI:10.1098/rspa.2015.0733}


\bibitem[Iserles00]{IMNZ:2000}
Iserles, HZ.~Munthe-Kaas, SP.~Nørsett and A.~Zanna,
\newblock Lie-group methods, 
\newblock Acta Numerica \textbf{9} (2000), pp~215-365. \href{http://doi.org/10.1017/S0962492900002154}{DOI:10.1017/S0962492900002154}.


\bibitem[Iserles00]{IZ:2000}
A.~Iserles and A.~Zanna,
\newblock On the Dimension of Certain Graded Lie Algebras Arising in Geometric Integration of Differential Equations. 
\newblock LMS Journal of Computation and Mathematics, \textbf{3} (2000), pp.~44-75. \href{http://doi:10.1112/S1461157000000206}{DOI:10.1112/S1461157000000206}.


\bibitem[Ito07]{IK2016}
K.~Ito and K.~Kunisch
\newblock Optimal bilinear control of an abstract Schrödinger equation
\newblock SIAM Journal on Control and OptimizationVol. \textbf{46} (2007), Iss.~1, \href{http://doi.org/10.1137/05064254X}{DOI:10.1137/05064254X}.


\bibitem[Jawecki23]{JS2023}
T.~Jawecki, P.~Singh, 
\newblock Unitary rational best approximations to the exponential function,
\newblock \href{https://arxiv.org/pdf/2312.13809}{arxiv:2312.13809}.


\bibitem[Kormann08]{KHK2008}
K.~Kormann, S.~Holmgren, H.O.~Karlsson,
\newblock Accurate time propagation for the Schrodinger equation with an explicitly time-dependent Hamiltonian. 
\newblock J Chem Phys. \textbf{128} (2008), no.~18:184101. \href{http://doi.org/10.1063/1.2916581}{DOI:10.1063/1.2916581}. 


\bibitem[McLachlan02]{LQ2002}
McLachlan RI and Quispel GRW,
\newblock Splitting methods. Acta Numerica. \textbf{11} (2002), pp.~341-434, \href{http://doi.org/10.1017/S0962492902000053}{DOI:10.1017/S0962492902000053}.



\bibitem[Marthinsen01]{MB:2001}
A.~Marthinsen and B.~Owren,
\newblock Quadrature methods based on the Cayley transform,
Applied Numerical Mathematics, \textbf{39} (2001), no.~3–4, pp.~403-413, \href{https://doi.org/10.1016/S0168-9274(01)00087-3}{DOI:10.1016/S0168-9274(01)00087-3}.


\bibitem[Magnus54]{Magnus1954}
W.~Magnus, 
\newblock On the exponential solution of differential equations for a linear operator \newblock Commun. Pure Appl. Math. \textbf{7} (1954), pp.~649-673, \href{https://doi.org/10.1002/cpa.3160070404}{DOI:10.1002/cpa.3160070404}.


\bibitem[Manop02]{Manolopoulos}
D.E.~Manolopoulos,
\newblock Derivation and reflection properties of a transmission-free absorbing potential,
\newblock J. Chem. Phys. \textbf{117} (2002), pp.~9552–9559, \href{https://doi.org/10.1063/1.1517042}{DOI:10.1063/1.1517042}. 


\bibitem[Moler03]{MV}
C.Moler and C.~Van Loan,
\newblock Nineteen Dubious Ways to Compute the Exponential of a Matrix, Twenty-Five Years Later,
\newblock SIAM Review \textbf{45} (2003), pp.~3-49    \href{https://doi.org/10.1137/S00361445024180}{DOI:10.1137/S00361445024180}.


\bibitem[Oteo91]{Oteo1991}
J.A.~Oteo, 
\newblock The Baker–Campbell–Hausdorff formula and nested commutator identities, 
\newblock J. Math. Phys. \textbf{32} (1991), pp.~419–424, \href{https://doi.org/10.1063/1.529428}{DOI:10.1063/1.529428}.


\bibitem[MK99]{MO1999}
H.~Munthe–Kaas and B.~Owren
\newblock Computations in a free Lie algebra
\newblock Phil. Trans. R. Soc. A. \textbf{357} (1999), pp.~957–981
\href{https://doi.org/10.1098/rsta.1999.0361}{DOI:10.1098/rsta.1999.0361}.


\bibitem[Singh19]{Pranav2019}
P.~Singh,
\newblock Sixth-order schemes for laser–matter interaction in the Schrödinger equation. 
\newblock J. Chem. Phys. 21 April 2019; \textbf{150} (15): 154111, \href{https://doi.org/10.1063/1.5065902}{DOI:10.1063/1.5065902}.


\bibitem[Singh24]{expsolve}
P.~Singh and D.~Goodacre,
\newblock expsolve: A differentiable numerical algorithms package for computational quantum mechanics, 
\newblock Zenodo (2024). 
\href{https://doi.org/10.5281/zenodo.13121390}{DOI:10.5281/zenodo.13121390}.



\bibitem[Saad03]{Saad2003}
Y.~Saad,
\newblock Iterative methods for sparse linear systems,
\newblock Society for Industrial and Applied Mathematics, 2003.


\bibitem[Reeth16]{VRTGS2016}
E.~Van Reeth, H.~Rafiney, M.~Tesch, S.J.~Glaser, D.~Sugny,
\newblock Optimizing mri contrast with b1 pulses using optimal control theory, 
\newblock in: Proc. IEEE Int. Symp. Biomed. Imaging, 2016, pp.~310–313, \href{https://doi.org/10.1109/ISBI.2016.7493271}{DOI:10.1109/ISBI.2016.7493271}


\bibitem[Wathen15]{Wathen}
AJ.~Wathen,
\newblock Preconditioning,
\newblock Acta Numerica. \textbf{24} (2015), pp.~329-376. \href{https://doi.org/10.1017/S0962492915000021}{DOI:10.1017/S0962492915000021}.


\end{thebibliography}

\end{document}